\documentclass{article}

\usepackage{authblk}

\usepackage[utf8]{inputenc}
\usepackage[T1]{fontenc}
\usepackage[english]{babel}
\usepackage{textcomp}
\usepackage{amsmath,amssymb,amsthm}
\usepackage{lmodern}
\usepackage[a4paper]{geometry}
\usepackage{xcolor, pict2e}
\usepackage{microtype}
\usepackage{caption}
\usepackage{listings}
\usepackage{multicol}
\usepackage{moreverb}
\usepackage{hyperref}
\hypersetup{pdfstartview=XYZ}
\usepackage{wrapfig}
\usepackage[sans]{dsfont}

\usepackage{tikz, pgfplots}
\usetikzlibrary{positioning}
\usepackage{amsmath,amssymb}
\usetikzlibrary{decorations.pathreplacing}

\usepackage{ytableau}

\usepackage{mathdots}

\usepackage{hyperref}
\hypersetup{
    colorlinks=true,
    linkcolor=blue,
    citecolor=magenta,
    urlcolor=blue,
    pdfborder={0 0 0}
}

\definecolor{fond}{rgb}{0.05,0.05,0.25}

\DeclareMathOperator{\Poiss}{Poiss}
\DeclareMathOperator{\dtv}{d_{TV}}

\newcommand{\dLtwo}{\ensuremath{\textup{d}_{L^2}}}

\DeclareMathOperator{\ch}{\mathrm{ch}}
\DeclareMathOperator{\Id}{\mathrm{Id}}

\DeclareMathOperator{\dgex}{\cE}

\DeclareMathOperator{\cyc}{\mathrm{cyc}}
\DeclareMathOperator{\Unif}{\mathrm{Unif}}
\DeclareMathOperator{\Bin}{\mathrm{Bin}}

\DeclareMathOperator{\Shift}{Shift}

\DeclareMathOperator{\Cay}{\mathrm{Cay}}

\newcommand{\Cvirt}{\ensuremath{C}_\textup{virt}}
\newcommand{\Cpart}{\ensuremath{C}_\textup{part}}

\newcommand{\rsym}{\ensuremath{r}_\textup{sym}}

\newcommand{\kS}{{\ensuremath{\mathfrak{S}}} }
\newcommand{\kA}{{\ensuremath{\mathfrak{A}}} }
\newcommand{\kE}{{\ensuremath{\mathfrak{E}}} }

\usepackage[labelfont=bf, margin=1cm]{caption}
\usepackage{enumitem}
\setlist[itemize,1]{nosep}
\setlist[enumerate,1]{itemsep=0pt,label=(\alph*)}

\usepackage{float}
\usepackage[caption = false]{subfig}
\usepackage{graphicx}

\newtheorem*{theorem*}{Theorem}
\newtheorem{theorem}{Theorem}[section]
\newtheorem{proposition}[theorem]{Proposition}
\newtheorem{lemma}[theorem]{Lemma}

\newtheorem{example}[theorem]{Example}
\newtheorem{definition}[theorem]{Definition}

\theoremstyle{remark}
\newtheorem{remark}{Remark}[section]

\usepackage{todonotes}

\newcommand{\cC}{{\ensuremath{\mathcal C}} }

\newcommand{\cE}{{\ensuremath{\mathcal E}} }

\newcommand{\cT}{{\ensuremath{\mathcal T}} }

\newcommand{\bbE}{{\ensuremath{\mathbb E}} }

\newcommand{\bbZ}{{\ensuremath{\mathbb Z}} }

\newcommand{\bbR}{{\ensuremath{\mathbb R}} }
\newcommand{\bbC}{{\ensuremath{\mathbb C}} }
\newcommand{\bbP}{{\ensuremath{\mathbb P}} }

\newcommand{\ag}{\left\{ } 

\newcommand{\ad}{\right\} }

\newcommand{\cg}{\left[}
\newcommand{\cd}{\right]}
\newcommand{\pg}{\left(} 
\newcommand{\pd}{\right)}
\newcommand{\bg}{\left|}
\newcommand{\bd}{\right|}
\newcommand{\lf}{\left\lfloor}
\newcommand{\rf}{\right\rfloor}
\newcommand{\lc}{\left\lceil}
\newcommand{\rc}{\right\rceil}

\newcommand{\du}{{\ensuremath{\;:\;}}} 

\newcommand{\floor}[1]{\left\lfloor #1 \right\rfloor}

\DeclareMathOperator{\ST}{ST}

\DeclareMathOperator{\supp}{supp}

\DeclareMathOperator{\Conj}{Conj}

\DeclareMathOperator{\sgn}{sgn}

\makeatletter
\newcommand*\bigcdot{\mathpalette\bigcdot@{.5}}
\newcommand*\bigcdot@[2]{\mathbin{\vcenter{\hbox{\scalebox{#2}{$\m@th#1\bullet$}}}}}
\makeatother

\numberwithin{equation}{section}

\pgfplotsset{compat=1.18}

\begin{document}

\renewcommand{\theparagraph}{\thesubsection.\arabic{paragraph}} 
\title{Sharp character bounds and cutoff for symmetric groups}

\author[1]{Sam Olesker-Taylor}
\author[2]{Lucas Teyssier}
\author[3]{Paul Thévenin}
\affil[1]{University of Warwick, \texttt{oleskertaylor.sam@gmail.com}}
\affil[2]{University of British Columbia, \texttt{teyssier@math.ubc.ca}}
\affil[3]{Université d'Angers, \texttt{paul.thevenin@univ-angers.fr}}

\maketitle
\begin{abstract}
        We develop a flexible technique to bound the characters of symmetric groups, via the Naruse hook length formula, the Larsen--Shalev character bounds, and appropriate diagram slicings. It allows us to prove a uniform exponential character bound with optimal constant $1/2$.
        We furthermore prove sharp character bounds for conjugacy classes having a macroscopic number of fixed points, and deduce that the random walks on the associated Cayley graphs exhibit a total variation and $L^2$ cutoff.
\end{abstract}

\tableofcontents

\section{Introduction}

Fourier analysis on groups, also known as representation theory, is a rich topic connected to many areas of mathematics. Obtaining precise bounds on characters enables solving a variety of problems from mixing times, geometric group theory, word maps, covering numbers, and random maps. Strong character bounds were proved for finite groups of Lie type \cite{Gluck1993, LarsenShalevTiep2011Waring, BezrukavnikovLiebeckShalevTiep2018, LarsenTiep2024FiniteClassicalGroups}, finite classical groups \cite{GuralnickLarsenTiep2020characterlevels1,GuralnickLarsenTiep2024characterlevels2}, and symmetric groups \cite{Roichman1996, MullerSchlagePuchta2007precutoff, LarsenShalev2008, FeraySniady2011}.  
Of particular interest for applications are \textit{exponential} character bounds, that is, bounds on characters written as powers of the dimensions of irreducible representations. In this paper we prove sharp exponential character bounds for symmetric groups that are optimal for applications to mixing times.

\subsection{Main results}

\subsubsection{Bounds on characters}

Let $\kS_n$ be the symmetric group of index $n$, and denote by $\widehat{\kS_n}$ its set of irreducible representations. If $\sigma\in \kS_n$ and $\lambda \in \widehat{\kS_n}$, we denote the dimension of the irreducible representation $\lambda$ by $d_\lambda$, the associated character by $\ch^\lambda(\sigma)$, and the associated reduced character by $\chi^\lambda(\sigma) := \ch^\lambda(\sigma)/d_\lambda$.
If $\sigma \in \kS_n$, we denote its conjugacy class by $\cC(\sigma)$, its number of cycles of length $i$ by $f_i(\sigma)$, its total number of cycles by $\cyc(\sigma)$, and we set $f(\sigma) := \max(f_1(\sigma), 1)$. 

\medskip

The first exponential bound over a wide range of conjugacy classes was proved by Roichman \cite{Roichman1996}. He showed that for every $\delta>0$, there exists a constant $q>0$ such that uniformly over all irreducible representations $\lambda \in \widehat{\kS_n}$, and all $\sigma \in \kS_n$ with at least $\delta n$ fixed points, 
\begin{equation}\label{eq: uniform bound characters symmetric}
        \bg \chi^\lambda(\sigma) \bd \leq d_\lambda^{\pg - q+o(1) \pd\frac{\ln (n/f(\sigma))}{\ln n}}.
    \end{equation}
Later, Müller and Schlage-Puchta \cite{MullerSchlagePuchta2007precutoff} improved this to a uniform bound over all permutations, with an explicit constant $q=30$ (and $q=12$ if restricted to permutations with at least $n^{1-o(1)}$ fixed points).
Then, alongside celebrated character bounds that depend on the whole cycle structure of permutations and are sharp for fixed-point-free permutations, 
Larsen and Shalev \cite{LarsenShalev2008} proved that \eqref{eq: uniform bound characters symmetric} holds with constant $q = 1/2$ uniformly over permutations with at most $n^{1-o(1)}$ fixed points, which is sharp for fixed-point-free involutions. 
Recently, Lifschitz and Marmor \cite{LifschitzMarmor2024charactershypercontractivity} also proved a character bound similar to \eqref{eq: uniform bound characters symmetric} for permutations with at most $n/(\ln n)^{O(1)}$ cycles. 
Finally, in \cite{TeyssierThévenin2025virtualdegreeswitten} two of us extended the bounds of \cite{LarsenShalev2008}, and proved in particular that \eqref{eq: uniform bound characters symmetric} holds with constant $q=1/2$ uniformly for permutations with $o(n)$ fixed points. We complete this line of results, proving a uniform character bound for symmetric groups with optimal constant $1/2$.

\begin{theorem}\label{thm: uniform bound 1/2 intro}
    As $n\to \infty$, uniformly over all $\sigma \in \kS_n$ and $\lambda\in \widehat{\kS_n}$, we have
\begin{equation}
    \bg \chi^\lambda(\sigma) \bd \leq d_\lambda^{\pg - 1/2+o(1) \pd\frac{\ln (n/f(\sigma))}{\ln n}}.
\end{equation}
\end{theorem}

 Probabilistic intuition from mixing times suggests that sharper bounds should hold given mild constraints on the cycle structure, and moreover that the number of fixed points is the most crucial parameter in this setup. We prove optimal character bounds for permutations with order $n$ fixed points, which is the scale considered by Roichman \cite{Roichman1996}.

\begin{theorem}\label{thm: character bound with no error nor helping term}
Let $\delta>0$. There exists $n_0 = n_0(\delta)$ such that for every $n\geq n_0$, uniformly over all $\lambda \in \widehat{\kS_n}$ and $\sigma \in \kS_n$ such that $f(\sigma)\geq \delta n$, we have 
\begin{equation}
    \bg \chi^\lambda(\sigma) \bd \leq d_\lambda^{-\frac{\ln (n/f(\sigma))}{\ln n}}.
\end{equation}
\end{theorem}

Theorem \ref{thm: character bound with no error nor helping term} not only has the best possible constant $q=1$ in \eqref{eq: uniform bound characters symmetric}, but also does not have any error term. Actually, allowing the bound to depend on the shape of the representation $\lambda$, we can obtain even stronger bounds, as we show in Theorem \ref{thm: character bound with helping term}.

\subsubsection{Applications to mixing times}
Let $\cC$ be a conjugacy class of $\kS_n$, let $\sigma_1, \sigma_2,\ldots$ be i.i.d.\ random variables uniform on $\cC$, and let $X_t = \sigma_1\cdots\sigma_t$ for $t\geq 0$. The law of $X_t$ is then given by $\Unif_{\cC}^{*t}$, where $\Unif_E$ denotes the uniform measure on a set $E$. Another point of view is that $X_t$ is the position after $t$ steps of the simple random walk started at the identity permutation $\Id$ on the Cayley graph $\Cay(\kS_n, \cC)$, that is, on the graph with vertex set $\kS_n$ and edge set $\ag \ag \sigma, \sigma \tau\ad \du (\sigma, \tau)\in \kS_n \times \cC \ad$. 

Denote by $\kA_n$ the alternating group of index $n$. Denote also by $\mathfrak{E}(\cC, t)$ 
the set $\mathfrak{S}_n \backslash\mathfrak{A}_n$ if $\cC \subset \mathfrak{S}_n \backslash\mathfrak{A}_n$ and $t$ is odd, and $\mathfrak{A}_n$ otherwise; this is the coset of the alternating group $\kA_n$ on which $X_t$ lies.
Finally, we define the total variation and $L^2$ (coset) distances to stationarity by
\begin{equation}\label{eq: def dtv coset}
    \dtv(\cC, t) = \dtv\pg \Unif_{\cC}^{*t}, \Unif_{\kE_n(\cC, t)} \pd = \frac{1}{2}\sum_{\sigma \in \kE(\cC, t) } \bg \Unif_{\cC}^{*t}(\sigma) - \frac{2}{n!}\bd,
\end{equation}
and
\begin{equation}\label{eq: def dLtwo coset}
    \dLtwo(\cC, t) = \dLtwo\pg \Unif_{\cC}^{*t}, \Unif_{\kE_n(\cC, t)} \pd = \sqrt{\frac{n!}{2}  \sum_{\sigma \in \kE(\cC, t) } \bg \Unif_{\cC}^{*t}(\sigma) - \frac{2}{n!}\bd^2}.
\end{equation}

\medskip

In a landmark paper, Diaconis and Shahshahani \cite{DiaconisShahshahani1981} used representation theory to show a sharp phase transition, called \textit{cutoff}, for the conjugacy class $\cT^{(n)}$ of transpositions on $\kS_n$.  They showed that for any $\varepsilon>0$,
\begin{equation}
        \dtv\pg \cT^{(n)}, \lf (1-\varepsilon)\frac{1}{2}n\ln n\rf\pd \xrightarrow[n\to \infty]{} 1 \quad \text{ and } \quad  \dtv\pg \cT^{(n)}, \lc (1+\varepsilon)\frac{1}{2}n\ln n \rc \pd\xrightarrow[n\to \infty]{} 0.
    \end{equation}
This not only initiated the theory of mixing times, together with Aldous \cite{Aldous1983mixing, AldousDiaconis1986}, but also built a fruitful connection between probability and representation theory. 

\medskip

The phase transition result of Diaconis and Shahshashani \cite{DiaconisShahshahani1981} had several refinements and extensions to other conjugacy classes, that we discuss in Section \ref{s: more background and related results}.
We show that cutoff holds in total variation and in $L^2$ for all conjugacy classes with order $n$ fixed points.

Denote the set of conjugacy classes of $\kS_n$ by $\Conj(\kS_n)$, and set $\Conj^*(\kS_n) = \Conj(\kS_n)\backslash \ag \ag \Id \ad \ad$.
If $\cC\in \Conj(\kS_n)$, $\sigma \in \cC$, and $F$ is a class function (that is, a function $F\du \kS_n \to \bbC$ such that $F(\sigma_1) = F(\sigma_2)$ if $\sigma_1$ and $\sigma_2$ are in the same conjugacy class), we may write $F(\cC)$ for $F(\sigma)$. 
For $\cC\in \Conj^*(\kS_n)$, we set
\begin{equation}
    t_{\cC} = \frac{\ln n}{\ln(n/f(\cC))}.
\end{equation}

\begin{theorem}\label{thm: cutoff intro}
Let $\delta>0$ and $\varepsilon>0$. As $n\to \infty$, uniformly over all conjugacy classes $\cC \in \Conj^*(\kS_n)$ such that $f(\cC) \geq \delta n$, we have
  \begin{equation}
        \dtv(\cC, \lf (1-\varepsilon)t_{\cC}\rf) \xrightarrow[n\to \infty]{} 1 \quad \text{ and } \quad  \dtv(\cC, \lc (1+\varepsilon)t_{\cC}\rc) \xrightarrow[n\to \infty]{} 0,
    \end{equation}
and
\begin{equation}
        \dLtwo(\cC, \lf (1-\varepsilon)t_{\cC}\rf) \xrightarrow[n\to \infty]{} \infty \quad \text{ and } \quad  \dLtwo(\cC, \lc (1+\varepsilon)t_{\cC}\rc) \xrightarrow[n\to \infty]{} 0.
    \end{equation}
\end{theorem}
This is the first $L^2$ cutoff over a wide range of conjugacy classes, for any group. We also prove a refined $L^2$ bound in Proposition \ref{prop: improved ell two upper bound}.

We conjecture that the results of Theorems \ref{thm: character bound with no error nor helping term} and \ref{thm: cutoff intro} extend to all conjugacy classes with $\omega(\sqrt{n})$ fixed points.
The results for conjugacy classes $\cC$ of $\kS_n$ with $O(\sqrt{n})$ fixed points must however be different, due to cancellations of transpositions that can prevent $\Unif_{\cC}^{*2}$ to be close to $\Unif_{\kA_n}$ in total variation.

\subsection{Additional background and related results}\label{s: more background and related results}

Since the foundational work of Diaconis and Shahshahani \cite{DiaconisShahshahani1981},  mixing times for random walk on groups have been widely studied, and the connections between mixing times and representation theory have kept growing stronger. 
Diaconis \cite{LivreDiaconis1988, Diaconis1996CutoffFiniteMarkovChains} led further discussions about this topic, and wrote the following about character bounds leading to cutoff for conjugacy invariant walks on $\kS_n$ in \cite{Diaconis2003RandomWalksGroupsCharactersGeometry}:  “This is closely
related to earlier work of Roichman \cite{Roichman1996} who derived useful bounds on character
ratios for large conjugacy classes in $\kS_n$. Alas his work has not been pushed through
to give sharp bounds in probability problems. This is a hard but potentially fruitful area of study.” 

\medskip

The phase transition result of Diaconis and Shahshahani \cite{DiaconisShahshahani1981} had refinements and alternative proofs \cite{SaloffCosteZuniga2008, White2019}, and several extensions to other conjugacy classes of $\kS_n$: conjugacy classes with support $\leq 6$ and $7$-cycles in \cite{Roussel2000}, $k$-cycles with $k = O(1)$ in \cite{BerestyckiSchrammZeitouni2011}, $k$-cycles with $k = o(n)$ in \cite{Hough2016}, and conjugacy classes with support $o(n)$ in \cite{BerestyckiSengul2019}. 
Also, soft techniques relying on entropy \cite{Salez2024} recently enabled recovering that conjugacy classes with support $n^{o(1)}$ exhibit cutoff \cite{PedrottiSalez2025newcriterion}. Theorem \ref{thm: cutoff intro} proves that cutoff holds for all conjugacy classes with order $n$ fixed points, both in total variation and in $L^2$.

For conjugacy classes with larger support, Lulov and Pak \cite{LulovPak2002} studied long $k$-cycles, with $k > n/2$, and Lulov \cite{Lulovthesis1996} understood the mixing time of conjugacy classes with a unique cycle length ($n/m$ cycles of length $m$), using the character bounds of Fomin and Lulov
\cite{FomLul95}. Then Larsen and Shalev \cite{LarsenShalev2008} proved that fixed-point-free conjugacy classes mix in 2 or 3 steps, and building on the work of Larsen and Shalev, two of us characterized in \cite{TeyssierThévenin2025virtualdegreeswitten} which fixed-point-free conjugacy classes mix in 2 steps. 

What happens within the phase transition is also now better understood. In \cite{Teyssier2020}, the second named author found a character identity involving Poisson--Charlier polynomials, that enabled him to compute the \textit{cutoff profile}. For fixed $a\in \bbR$,
the total variation distance to stationarity after $\frac{1}{2}(n\ln n + an)$ steps converges to $\dtv\pg \Poiss(1), \Poiss\pg 1+e^{-a} \pd \pd$, where $\Poiss(\alpha)$ denotes the Poisson distribution with parameter $\alpha >0$.
This was extended to cycles of length $k=o(n)$ in \cite{NestoridiOlesker-Taylor2022limitprofiles} and another proof for transpositions was also recently found
\cite{JainSawhney2024transpositionprofileotherproof}. This will be further extended to conjugacy classes with order $n$ fixed points in a follow up paper \cite{Teyssier2025ProfilesConjugacyInvariant}, using the character bounds proved in the present paper.

\medskip

Exponential character bounds as in \eqref{eq: uniform bound characters symmetric} are equivalent to mixing time results in $L^2$ (which imply upper bounds in total variation), and obtaining mixing time bounds was already one of the main aims for several of the aforementioned general representation theoretic results, including \cite{Roichman1996,MullerSchlagePuchta2007precutoff,LarsenShalev2008}. Roichman \cite{Roichman1996} found the correct order of magnitude of the mixing time for conjugacy classes of $\kS_n$ with order $n$ fixed points. Müller and Schlage-Puchta’s uniform character bound \cite{MullerSchlagePuchta2007precutoff} showed that $\dLtwo(\cC, 31\frac{\ln n}{\ln(n/f(\cC))})\to 0$ as $n\to \infty$, uniformly over all conjugacy classes $\cC \ne \ag \Id \ad$. Finally, Larsen and Shalev \cite{LarsenShalev2008} gave precise bounds for permutations with few fixed points, and in particular showed that $\dLtwo(\cC, 3)\to 0$ as $n\to \infty$, uniformly all fixed-point-free conjugacy classes $\cC$. 

\medskip

Several related conjectures were made. Roichman’s conjecture \cite[Conjecture 6.7]{Roichman1996} is strongly related to Theorem \ref{thm: character bound with no error nor helping term}. Character bounds are now usually written in exponential form, since this interacts well with bounds on the Witten zeta function, as explained in Section \ref{s: a classical technique}, but the conjectured bound is similar for applications to mixing times. However the conjecture was too optimistically stated to hold uniformly over all conjugacy classes, and \cite{MullerSchlagePuchta2007precutoff} observed that it fails to hold for some fixed-point-free conjugacy classes. 

Roussel \cite[Conjecture 1.1]{Roussel2000} conjectured that for any $\delta>0$, there exists a constant $C>0$ such that for any $a\in \bbR$,  as $n\to \infty$, uniformly over all conjugacy classes $\cC$ whose support $\supp(\cC)$ satisfies $2\leq \supp(\cC) \leq (1-\delta)n$, we have $\dtv(\cC, t_{\cC}^*\pg 1 + a/\ln n \pd)\leq Ce^{-a}$, where $t_{\cC}^* := \frac{n}{\supp(\cC)}\ln n$. We always have $t_{\cC}^*\geq t_{\cC}$, therefore Proposition \ref{prop: improved ell two upper bound} proves a stronger version of Roussel’s conjecture.

Saloff-Coste conjectured, below Conjecture 9.3 in \cite{Saloff-Coste2004RandomWalksOnFiniteGroups}, that the $L^2$ \textit{cutoff time} for the non-lazy chain is $t^*_{\cC}$, for conjugacy classes $\cC$ such that $\supp(\cC) = o(n)$. The part about the lazy chain was proved to not hold in \cite{SaloffCosteZuniga2008}, due to interactions between the $L^2$ distance to stationarity and laziness. Proposition \ref{prop: improved ell two upper bound} proves a stronger version of Saloff-Coste’s conjectured upper bound for the non-lazy chain: this is proved replacing the time $t^*_{\cC} = \frac{n}{\supp(\cC)}\ln n$ by the time $t_\cC = \frac{\ln n}{\ln(n/f(\cC))}$ and with the milder support size condition $n-\supp(\cC) \gtrsim n$ instead of $\supp(\cC) = o(n)$.

\medskip

Proofs via representation theory present some benefits. One is that bounds on characters enable solving other problems; see for instance the applications of \cite{LarsenShalev2008} on covering numbers, normal subsets, and word maps, and the results of \cite{Gamburd2006} and \cite{ChmutovPittel2016} on random maps.
Another advantage is that character bounds give $L^2$ bounds for mixing times, which can be used in comparison techniques \cite{DiaconisSaloff-Coste1993comparisongroups, DiaconisSaloff-Coste1993comparisonreversiblechains} to bound the mixing times of other Markov chains. 

\medskip

The correct order of the mixing time is now understood in a wide framework. Larsen and Tiep \cite[Propositions 8.1 and 8.2]{LarsenTiep2024FiniteClassicalGroups}
recently showed that simple random walks on $\Cay(G,\cC)$ are close to being uniformly distributed after order $\frac{\ln |G|}{\ln |\cC|}$ steps (and are not before), with implicit universal constants that apply to all (non-trivial) conjugacy classes $\cC$ of all quasi-simple groups $G$ of Lie type.
It would be very interesting to obtain cutoff results also in this setting. 

\medskip

Finally, character bounds with a different flavour, not relying on the Murnaghan--Nakayama rule, were proved by Rattan and Śniady \cite{RattanSniady2008} using a character formula from Frobenius, and by Féray and Śniady \cite{FeraySniady2011} via Stanley’s character formula. These bounds are precise if representations have short rows and columns in their coding as Young diagrams, and permutations have small $O(\sqrt{n})$. This led to interesting applications to asymptotic combinatorics \cite{DousseFéray2019skewdiagramscharacters}. 

\subsection{General strategy}

Our approach is mostly combinatorial, but strongly influenced by probabilistic intuition. For conjugacy invariant random walks on symmetric groups, the limiting observable is usually the number of fixed points, which should therefore also be the most important parameter for characters bounds. The rest of the cycle structure should play a minor role, as long as the associated random walk needs at least 3 steps to mix, which is always the case if the number of fixed points grows faster than $\sqrt{n}$. 

With this in mind, we tried to split the contribution of the fixed points from the rest of the cycle structure in the Murnaghan--Nakayama rule. This led us to the iterated branching rule 
\begin{equation}\label{eq:branching rule intro non renormalisée}
    \ch^\lambda(\sigma) = \sum_{\mu \subset_{\vdash k} \lambda} \ch^\mu(\sigma^*) d_{\lambda\backslash \mu},
\end{equation}
where $k$ is the support size of $\sigma$, and $\sigma^*\in \kS_k$ has the same cycle structure as $\sigma$, except that it has no fixed point. We give more details on this branching rule in Section \ref{s:An iterated branching rule}.
The fixed-point-free part (the term $\ch^\mu(\sigma^*)$) is handled with the Larsen--Shalev character bounds \cite{LarsenShalev2008} (and their sharpening from \cite{TeyssierThévenin2025virtualdegreeswitten} for the limiting cases), while the fixed point part (the term $d_{\lambda\backslash \mu}$, which is the number of standard tableaux of the skew diagram $\lambda \backslash \mu$) is bounded using the Naruse hook length formula \cite{Naruse2014}.
This converts the problem of bounding characters into bounding the number of standard tableaux of some shapes of diagrams.
\medskip

Although initially designed with applications to mixing times in mind, our technique and some of our intermediate estimates are also applicable at other scales and for other problems. In a follow up paper \cite{Teyssier2025SkewDimensionsAndCharacters}, the second named author improves the character bounds of \cite{RattanSniady2008, FeraySniady2011} for permutations having a large support. Our technique also enables simple estimation of the characters of small conjugacy classes for balanced diagrams.

\subsection{Notation}\label{s: notation}

Recall that we denote the symmetric group of order $n$ by $\mathfrak{S}_n$ and the set of its irreducible representations by $\widehat{\mathfrak{S}_n}$. For convenience, the pieces of notation $\lambda \in \widehat{\kS_n}$ and $\lambda \vdash n$ will be used interchangeably to describe that $\lambda$ is an irreducible representation of $\kS_n$, the associated integer partition of the integer $n$, or the associated Young diagram of size $n$.

Throughout the paper, we refer to a box, i.e.\ a unit square in the plane whose vertices have integer coordinates, by the horizontal and vertical coordinates of its top-right vertex. By extension, we see sets of boxes as subsets of $\bbZ^2$. In the case of a Young diagram $\lambda$, by default, we anchor its bottom-left corner at $(0,0)$; hence $(i,j)$ denotes the $i$-th leftmost and $j$-th bottommost box in $\lambda$, and in particular $(1,1)$ denotes the bottom-left box of $\lambda$.
We also write $\mu \subset_{\vdash k} \lambda$ when $\mu$ is a sub Young diagram of $\lambda$ of size $k$, that is, a set of boxes in $\lambda$ containing $(1,1)$ and which is itself a Young diagram.

\medskip

\noindent We make use of the following notation.
Let $E\subset \bbZ^2$, and $k \in \bbZ$.

\begin{itemize}
    \item $E_k$ is the $k$-th row of $E$, that is, $\{ (i,j)\in E \mid i=k \}$,
    \item $E'_k$ is the $k$-th column of $E$, that is, $\{ (i,j)\in E \mid j=k \}$,
    \item $E^k$ is the $k$-th hook of $\lambda$, that is, $\{ (i,j)\in \lambda \mid \min(i,j) =  k \}$,
    \item $u_k(E)$ is the $k$-th column truncated at 2, that is, $\{ (i,j)\in E \mid i \geq 2 , j=k \}$.
\end{itemize}
\medskip

\noindent We also extend the notation for Frobenius coordinates to sets of boxes as follows.

\begin{itemize}
    \item $a_k(E)$ denotes the set $\{ (i,j) \in E \mid i=k, j > k \}$,
    \item $b_k(E)$ denotes the set $\{ (i,j) \in E \mid i>k, j = k \}$,
    \item $\delta_k(E) = \ag (k,k) \ad \cap E$ and $\delta(E) = \cup_{i\in\bbZ} \delta_i(E)=\ag (i,i) \mid i \in \bbZ \ad \cap E$ is the diagonal of $E$,
    \item $a_k^+(E)$ denotes $\{ (i,j) \in E \mid i=k, j \geq k \}$, that is, $a_k^+(E) = a_k(E) \cup \delta_k(E)$.
\end{itemize}

\medskip

\noindent It is often convenient to consider unions of parts of diagrams. If $\Diamond\in \ag \geq, \leq, >, < \ad$, we denote $E_{\Diamond k} = \cup_{i\Diamond k} E_i$, $E^{\Diamond k} = \cup_{i\Diamond k} E^i$, and  $u_{\Diamond k} = \cup_{i\Diamond k} u_i$.

\medskip

If $u = (i_0,j_0)\in \bbZ^2$, we denote the infinite hook in $\bbZ^2$ whose corner is $u$ by $h_u$, that is, $\{ (i,j)\in \bbZ^2 \mid [i = i_0 \text{ and } j\geq j_0] \text{ or } [i \geq i_0 \text{ and } j= j_0] \}$,
and define the \textit{hook length} of $u$ in $E$ by
\begin{equation}\label{eq: hook lengths in the plane}
    H(E , u) = \bg E \cap h_u \bd.
\end{equation}

\noindent If $E$ and $F$ are subsets of $\bbZ^2$, we define
\begin{equation}\label{eq: hook lengths and products generalized H E F}
    H(E, F) = \prod_{u\in F} H(E, u).
\end{equation}

\noindent If $i,j\in \bbZ$, we define
\begin{equation}
    \Shift_{i,j}(E) := E + (i,j) =  \ag (x+i,y+j) \mid (x,y)\in E\ad.
\end{equation}

\noindent We call a set of boxes $E$ a \textit{row} (resp. a \textit{column}) diagram of length $\ell$ if $E = \Shift_{i,j}([\ell])$ (resp. $\Shift_{i,j}([1^\ell])$) for some $i,j\in \bbZ$.

\medskip

In order to lighten the notation, we may use the same pieces of notation for sub-diagrams and their sizes, and add at times brackets and absolute values to remove the ambiguity. For example we may write $\lambda_1$ or $[\lambda_1]$ both for the set of boxes on the first row of a diagram $\lambda$, and  $\lambda_1$ or $|\lambda_1|$ for its size.

Let $\lambda$ be a Young diagram. We set $\rsym(\lambda) = n - \max(\lambda_1, \lambda_1')$. If $\lambda_1 \geq \lambda_1'$, then $\rsym(\lambda) = n-\lambda_1$, so $\rsym$ is a symmetrized version of the notion of \textit{level} defined in \cite{KleshchevLarsenTiep2022levelsymmetricgroups}, which we also refer to as level. 

Finally, we use the following asymptotic notation. We write $f(n) = O(g(n))$ or $f(n) \lesssim g(n)$ or $g(n) = \Omega(f(n))$ if there exists a constant $C>0$ such that $|f(n)| \leq C |g(n)|$ for all $n$ large enough, and $f(n) = o(g(n))$ or $f(n) \lll g(n)$
or $g(n) = \omega(f(n))$ if $f(n)/g(n) \xrightarrow[n\to \infty]{}0$. To emphasize the positivity of some of these quantities, we may write $f(n) = \omega_+(g(n))$ (resp. $f(n) = \Omega_+(g(n))$) if $f(n) = \omega(g(n))$ (resp. $f(n) = \Omega(g(n))$) and $f(n)>0$ for $n$ large enough.

\section{Preliminaries}\label{s: preliminaries}

We assume familiarity with the representation theory of symmetric groups and refer to  \cite{LivreMéliot2017RepresentationTheoryofSymmetricGroups} for a complete introduction. In this section we define most concepts that are needed in our proofs, with a special emphasis on the Naruse hook length formula.

\subsection{Skew diagrams and standard tableaux}

Let $\lambda$ and $\mu$ be Young diagrams such that $\mu \subset \lambda$. The \textit{skew diagram} $\lambda\backslash \mu$ is defined as the diagram containing the boxes that are in $\lambda$ and not in $\mu$. In particular, it is not necessarily connected.

\begin{example}
    Let $\lambda = (7,5,4,1)$ and $\mu = (5,3,2)$. Then the skew diagram $\lambda\backslash \mu$ is the diagram
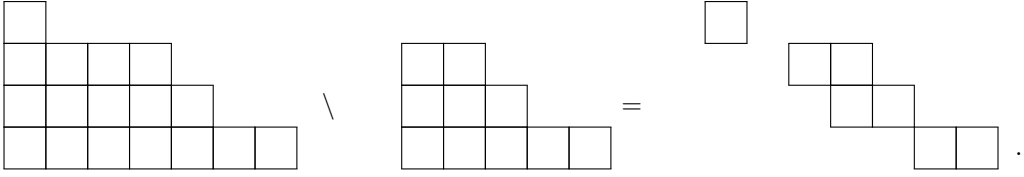
\begin{center}
\begin{ytableau}
    \none   & \\
    \none   &&&& \\
    \none   &&&&& \\
    \none   &&&&&&& 
\end{ytableau}
\begin{ytableau}
    \none  \\
    \none  \\
    \none[\backslash] \\
    \none
\end{ytableau}
\begin{ytableau}
    \none \\
    \none   && \\
    \none   &&& &\none&\none&\none[=] \\
    \none   &&&&& 
\end{ytableau}
\begin{ytableau}
    \none   & \\
    \none   &\none&\none&& \\
    \none   &\none&\none&\none&& \\
    \none   &\none&\none&\none&\none&\none&& &\none[.]
\end{ytableau}
\end{center}    
\end{example}

Let $p$ be the number of boxes of $\lambda \backslash \mu$. A standard tableau of the skew Young diagram $\lambda \backslash \mu$ is a bijection $B\du \lambda \backslash \mu \mapsto [p] := \ag 1,\ldots, p\ad$, such that $B(u)\leq B(v)$ for every $u,v\in \lambda \backslash \mu$ satisfying $v \in \ag \Shift_{1,0}(u), \Shift_{0,1}(u)\ad$. It is usually represented in the skew diagram $\lambda\backslash \mu$ writing the value $B(u)$ in each box $u\in \lambda\backslash\mu$. The set of standard tableaux of $\lambda \backslash \mu$ is denoted $\ST(\lambda \backslash \mu)$, and if $\mu$ is the empty diagram we write $\ST(\lambda)$ for $\ST(\lambda \backslash \emptyset)$: this is the set of standard tableaux of the Young diagram $\lambda$.

\begin{example}
Let $\lambda =(3,3,1)$ and $\mu = (2)$. $\ST(\lambda\backslash \mu)$ consists of the following tableaux:
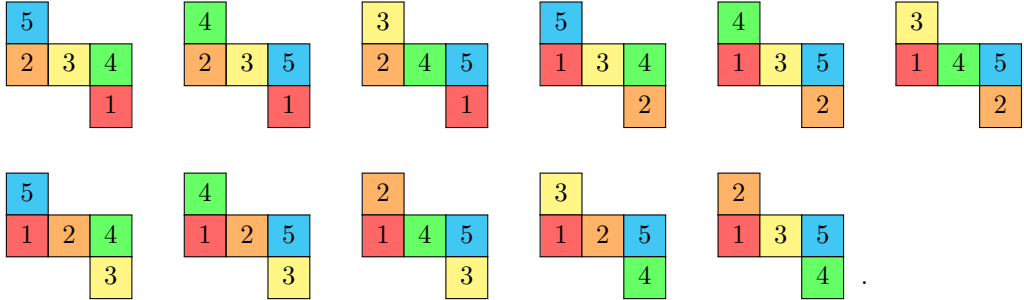
\begin{center}
     \begin{ytableau}
     \none &*(cyan!60)5 \\
     \none &*(orange!60)2&*(yellow!60)3&*(green!60)4\\
     \none &\none  &\none &*(red!60)1
\end{ytableau}
\begin{ytableau}
     \none &*(green!60)4 \\
     \none &*(orange!60)2&*(yellow!60)3&*(cyan!60)5\\
     \none &\none  &\none &*(red!60)1
\end{ytableau}
\begin{ytableau}
     \none &*(yellow!60)3 \\
     \none &*(orange!60)2&*(green!60)4&*(cyan!60)5\\
     \none &\none  &\none &*(red!60)1
\end{ytableau}  
\begin{ytableau}
     \none &*(cyan!60)5 \\
     \none &*(red!60)1&*(yellow!60)3&*(green!60)4\\
     \none &\none  &\none &*(orange!60)2
\end{ytableau}
\begin{ytableau}
     \none &*(green!60)4 \\
     \none &*(red!60)1&*(yellow!60)3&*(cyan!60)5\\
     \none &\none  &\none &*(orange!60)2
\end{ytableau}
\begin{ytableau}
     \none &*(yellow!60)3 \\
     \none &*(red!60)1&*(green!60)4&*(cyan!60)5\\
     \none &\none  &\none &*(orange!60)2
\end{ytableau}  

\begin{ytableau}
    \none \\
     \none &*(cyan!60)5 \\
     \none &*(red!60)1&*(orange!60)2&*(green!60)4\\
     \none &\none  &\none &*(yellow!60)3
\end{ytableau}
\begin{ytableau}
\none \\
     \none &*(green!60)4 \\
     \none &*(red!60)1&*(orange!60)2&*(cyan!60)5\\
     \none &\none  &\none &*(yellow!60)3
\end{ytableau}
\begin{ytableau}
\none \\
     \none &*(orange!60)2 \\
     \none &*(red!60)1&*(green!60)4&*(cyan!60)5\\
     \none &\none  &\none &*(yellow!60)3
\end{ytableau}  
\begin{ytableau}
\none \\
     \none &*(yellow!60)3 \\
     \none &*(red!60)1&*(orange!60)2 &*(cyan!60)5\\
     \none &\none  &\none &*(green!60)4 
\end{ytableau} 
\begin{ytableau}
\none \\
     \none &*(orange!60)2 \\
     \none &*(red!60)1&*(yellow!60)3 &*(cyan!60)5\\
     \none &\none  &\none &*(green!60)4 &\none[.] 
\end{ytableau}
\begin{ytableau}
\none \\
     \none &\none \\
     \none &\none&\none\\
     \none &\none  &\none 
\end{ytableau}

\end{center}
    
\end{example}

\subsection{Excited diagrams and the Naruse hook length formula}

The Naruse hook length formula, discovered by Naruse \cite{Naruse2014}, allows computing the number of standard tableaux $|\ST(\lambda\backslash \mu)|$ of a skew diagram $\lambda\backslash \mu$, generalizing the classical hook length formula for Young diagrams from Frame, Robinson, and Thrall \cite{FrameRobinsonThrall1954}. It involves sums of hook products over \textit{excited diagrams}, a notion introduced by Ikeda and Naruse \cite{IkedaNaruse2009exciteddiagrams}.
The Naruse hook length formula was proved algebraically and combinatorially by Morales, Pak, and Panova \cite{MoralesPakPanova2018I}, and a bijective proof was found by Konvalinka \cite{Konvalinka2020naruse}.

\label{ssec:Naruse hook length formula}
\subsubsection{Excited diagrams}

Let $\lambda \subset \bbZ^2$ be a (shift of a) Young diagram, and $E$ be a subset of $\lambda$.
We say that a box $u \in E$ of coordinates $(i,j)$ is \textit{free} if the box of coordinates $(i+1,j+1)$ is in $\lambda$ and none of the boxes $(i+1,j), (i,j+1)$, and $ (i+1,j+1)$ is in $E$. We say that $u$ is \textit{blocked} if it is not free.

\begin{example}
Let $\lambda = [5,5,5,2]$ and $\mu = [3,2,1,1]$.
In the figure below, we colour the boxes of $\mu \subset \lambda$ in green if they are free and in red if they are blocked:

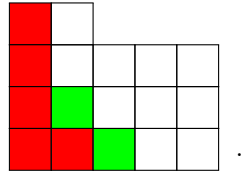
\begin{center}
\begin{ytableau}
    \none   &*(red)& \\
    \none   &*(red)&&&& \\
    \none   &*(red)&*(green)&&&\\
    \none   &*(red)&*(red)&*(green)&& &\none[.]
\end{ytableau}
\end{center}
\end{example}

An \textit{excitation} consists of moving one of the free boxes of $E$ in $\lambda$ towards north-east. If we apply an excitation, we say that we \textit{excite} (a box of) $E$ in $\lambda$.
A subset of the boxes of $\lambda$ that can be obtained from $E\subset \lambda$ by applying a (possibly empty) sequence of excitations to $E$ is called an \textit{excited diagram} of $E$ in $\lambda$. We denote by $\dgex(\lambda, E)$ the set of excited diagrams of $E$ in $\lambda$.

\begin{example}
Let $\lambda = [3,3,3]$ and $\mu = [1]$. Then $\dgex(\lambda, \mu)$ consists of three diagrams
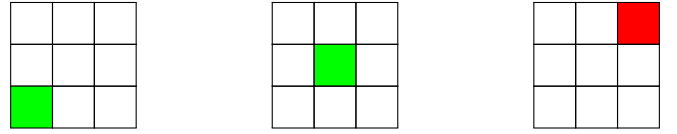
\begin{center}
\begin{ytableau}
    \none   &&& \\
    \none   &&& \\
    \none   &*(green)&& &\none &\none
\end{ytableau}
\begin{ytableau}
    \none   &&& \\
    \none   &&*(green)& \\
    \none   &&& &\none &\none
\end{ytableau}
\begin{ytableau}
    \none   &&&*(red) \\
    \none   &&& \\
    \none   &&& &\none[.]
\end{ytableau}
\end{center}
Here, as before and below, the green boxes are free and the red are blocked.
In addition, $\dgex(\lambda,\{ (2,2) \})$ consists of the last two diagrams.
\end{example}

\begin{example}
Let $\lambda = [3,3,3]$ and $\mu = [2,1]$. Then $\dgex(\lambda, \mu)$ consists of the five diagrams
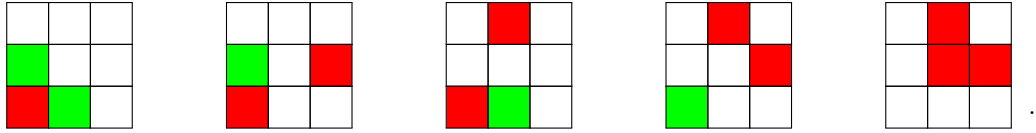
\begin{center}
\begin{ytableau}
    \none   &&& \\
    \none   &*(green)&& \\
    \none   &*(red)&*(green)& &\none 
\end{ytableau}
\begin{ytableau}
    \none   &&& \\
    \none   &*(green)&&*(red) \\
    \none   &*(red)&& &\none 
\end{ytableau}
\begin{ytableau}
    \none   &&*(red)& \\
    \none   &&& \\
    \none   &*(red)&*(green)& &\none 
\end{ytableau}
\begin{ytableau}
    \none   &&*(red)& \\
    \none   &&&*(red) \\
    \none   &*(green)&& &\none 
\end{ytableau}
\begin{ytableau}
    \none   &&*(red)& \\
    \none   &&*(red)&*(red) \\
    \none   &&& &\none[.]
\end{ytableau}
\end{center}
In the 4-th diagram, the corner box is now free, after the two other (initially green) boxes were moved away from it.
\end{example}

\begin{example}
Let $\lambda = [3,3,3,3,3]$ and $\mu = [1,1,1]$. Then $\dgex(\lambda, \mu)$ consists of the ten diagrams
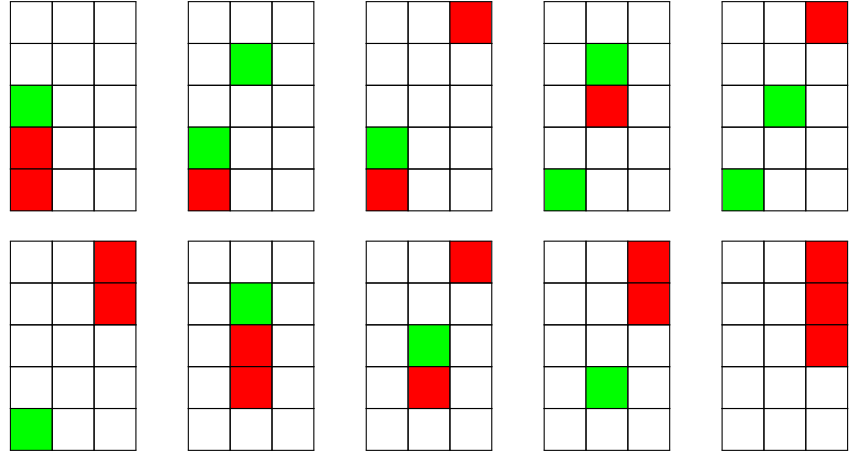
\begin{center}
\begin{ytableau}
    \none &&& \\
    \none &&& \\
    \none &*(green)&& \\
    \none &*(red)&& \\
    \none &*(red)&&
\end{ytableau}
\begin{ytableau}
    \none &&& \\
    \none &&*(green)& \\
    \none &&& \\
    \none &*(green)&& \\
    \none &*(red)&&
\end{ytableau}
\begin{ytableau}
    \none &&&*(red) \\
    \none &&& \\
    \none &&& \\
    \none &*(green)&& \\
    \none &*(red)&& 
\end{ytableau}
\begin{ytableau}
    \none &&& \\
    \none &&*(green)& \\
    \none &&*(red)& \\
    \none &&& \\
    \none &*(green)&&
\end{ytableau}
\begin{ytableau}
    \none &&&*(red) \\
    \none &&& \\
    \none &&*(green)& \\
    \none &&& \\
    \none &*(green)&& &\none
\end{ytableau}
\end{center}

\begin{center}
\begin{ytableau}
    \none &&&*(red) \\
    \none &&&*(red) \\
    \none &&& \\
    \none &&& \\
    \none &*(green)&&
\end{ytableau}
\begin{ytableau}
    \none &&& \\
    \none &&*(green)& \\
    \none &&*(red)& \\
    \none &&*(red)& \\
    \none &&&
\end{ytableau}
\begin{ytableau}
    \none &&&*(red) \\
    \none &&& \\
    \none &&*(green)& \\
    \none &&*(red)& \\
    \none &&&
\end{ytableau}
\begin{ytableau}
    \none &&&*(red) \\
    \none &&&*(red) \\
    \none &&& \\
    \none &&*(green)& \\
    \none &&&
\end{ytableau}
\begin{ytableau}
    \none &&&*(red) \\
    \none &&&*(red) \\
    \none &&&*(red) \\
    \none &&& \\
    \none &&& &\none[.]
\end{ytableau}
\end{center}
\end{example}

\begin{example}
Let $\lambda = [8,3,3,1]$ and $\mu = [6,2,1]$. Then $\dgex(\lambda, \mu)$ consists of the two diagrams
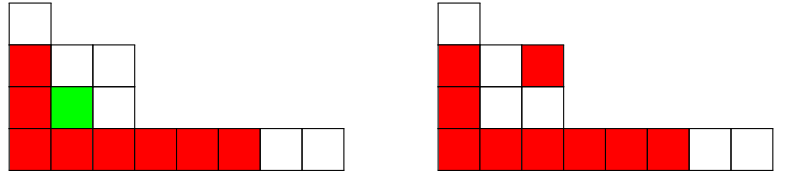
\begin{center}
\begin{ytableau}
    \none   & \\
    \none   &*(red)&& \\
    \none   &*(red)&*(green)&\\
    \none   &*(red)&*(red)&*(red)&*(red)&*(red)&*(red)&& &\none[]
\end{ytableau}
\begin{ytableau}
    \none   & \\
    \none   &*(red)&&*(red) \\
    \none   &*(red)&&\\
    \none   &*(red)&*(red)&*(red)&*(red)&*(red)&*(red)&& &\none[.]
\end{ytableau}
\end{center}
\end{example}

\subsubsection{The Naruse hook length formula}

Recall that if $\lambda$ is an irreducible representation of $\kS_n$ for some $n$, then $d_\lambda = |\ST(\lambda)|$, and that the hook length formula can be written as 
\begin{equation}\label{eq:hook length formula}
    |\ST(\lambda)| = \frac{|\lambda|!}{H(\lambda,\lambda)}.
\end{equation}
If $\lambda$ and $\mu$ are partitions and $\mu \subset \lambda$ we denote, by analogy with the hook length formula, 
\begin{equation}
    d_{\lambda \backslash \mu} := |\ST(\lambda\backslash \mu)|,
\end{equation}
and call it \textit{dimension} of $\lambda\backslash \mu$ by extension.
If $u$ is a box in $\lambda$, we denote by $H(\lambda, u)$ its hook length. If $E \subset \lambda$ is a subset of the boxes of $\lambda$ recall that we denote by $H(\lambda, E) = \prod_{u\in E} H(\lambda, u)$ the product of the hook lengths (calculated in $\lambda$) of the boxes in $E$.

In \cite[Theorem 1.2]{MoralesPakPanova2018I}, the Naruse hook length formula is written as follows: for any partitions $\lambda$ and $\mu \subset \lambda$, we have
\begin{equation}\label{eq:naruse as in MPP}
    d_{\lambda\backslash \mu} = |\lambda\backslash\mu|! \sum_{E\in \cE(\lambda,\mu)} \frac{1}{H(\lambda, \lambda \backslash E)}.
\end{equation}
This is a generalization of the hook length formula, since \eqref{eq:hook length formula} is exactly \eqref{eq:naruse as in MPP} applied with $\mu = \emptyset$.

For our purposes, it is convenient to rewrite this formula differently.
For any partition $\lambda$ and any subset $\mu \subset \lambda$, define
\begin{equation}\label{eq:def excited sums}
    S(\lambda, \mu) = \sum_{E \in \dgex(\lambda, \mu)} H(\lambda, E).
\end{equation}
We use the falling factorial notation: for any integers $0\leq p\leq m$,
\begin{equation}
    m^{\downarrow p} := \frac{m!}{p!} = m (m-1)\cdots(m-p+1).
\end{equation}
Combining \eqref{eq:naruse as in MPP} and the hook length formula, we can rewrite, if $\mu\subset \lambda$ is a Young diagram,
\begin{equation} \label{e:NHLF rewritten}
    \frac{d_{\lambda\backslash \mu}}{d_\lambda} = \frac{S(\lambda,\mu)}{|\lambda|^{\downarrow |\mu|}}.
\end{equation}
We also refer to this identity as the Naruse hook length formula.

\begin{example}
    Consider the diagrams $\lambda = [3,3,1]$ and $\mu = [2]$.
Then there are three excited diagrams of $\mu$ in $\lambda$:
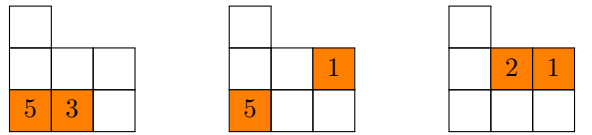
\begin{center}
    \begin{ytableau}
        \none & \\
        \none &&& \\
        \none &*(orange)5&*(orange)3& &\none 
    \end{ytableau}
    \begin{ytableau}
        \none & \\
        \none &&&*(orange)1 \\
        \none &*(orange)5&& &\none 
    \end{ytableau}
    \begin{ytableau}
        \none & \\
        \none &&*(orange)2&*(orange)1 \\
        \none &&& &\none[,] 
    \end{ytableau}
\end{center}
where we wrote the hook lengths of the orange boxes inside these boxes.
It follows that
\begin{equation}
    S(\lambda, \mu) = 5\cdot 3 + 5\cdot 1 + 2\cdot 1 = 22.
\end{equation}
Therefore, by \eqref{e:NHLF rewritten}, we have $\frac{d_{\lambda\backslash\mu}}{d_\lambda} = \frac{22}{7^{\downarrow 2}} = \frac{11}{21}$.
One can compute $d_\lambda = 21$ with the (standard) hook length formula, from which we deduce that $d_{\lambda\backslash \mu} = 11$.
\end{example}

Considering only some constraints on boxes can only increase the number of excited diagrams. We therefore have the following simple but very practical lemma.
\begin{lemma}\label{lem: fragmentation bound for excited sums}
    Let $\lambda$ be a Young diagram. Let $E\subset \lambda$ be a subset of the boxes of $\lambda$, and let $E_1, E_2,\ldots, E_p$ be disjoint subsets of $E$. Then
    \begin{equation}
      |\cE(\lambda, E)| \leq \prod_{i=1}^p |\cE(\lambda, E_i)| \quad \text{ and } \quad S(\lambda, E) \leq \prod_{i=1}^p S(\lambda, E_i).
    \end{equation}
\end{lemma}

\subsection{Sliced hook products}

In order to approximate hook products, it is convenient to consider the contribution of the hook lengths only from some parts of the diagrams. This idea appeared in the notion of virtual degree 
\begin{equation}\label{eq: def virtual degree}
    D(\lambda) = \frac{(n-1)!}{\prod_i a_i(\lambda)! b_i(\lambda)!}
\end{equation}
of a Young diagram $\lambda$ introduced by Larsen and Shalev \cite{LarsenShalev2008}. 
In \cite{TeyssierThévenin2025virtualdegreeswitten}, we extended this idea to other ways to slice diagrams and to hook products over subsets of the boxes of a diagram $\lambda$. Recall from \eqref{eq: hook lengths and products generalized H E F} the definition of $H(E, F)$ for $E,F\subset \bbZ^2$. \textit{Sliced hook products} in a diagram $\lambda$ are important in the proofs of the present paper. We recall the definition from \cite[Section 3]{TeyssierThévenin2025virtualdegreeswitten}.

\begin{definition}
    Let $\lambda$ be a Young diagram and $P = \ag \nu_1, \nu_2, \ldots, \nu_r\ad$ be a set partition of $\lambda$, that is, the sets $\nu_i$ are non-empty, disjoint, and their union is $\lambda$. Let $E$ be a subset of the boxes of $\lambda$. The
\textbf{sliced hook product} over the boxes of $E$ in the diagram $\lambda$ is defined as 
\begin{equation}
    H^{*P}(\lambda, E) := \prod_{i= 1}^r H(\nu_i, \nu_i\cap E).
\end{equation}
\end{definition}
 
Particularly important for the present paper are
\begin{itemize}
    \item the stairs slicing as in Figure \ref{fig:stairs slicing}, corresponding to the partition $P = \cup_{i\geq 1} \ag a_i^+(\lambda), b_i(\lambda) \ad$ of $\lambda$ (recall that the $a_i^+(\lambda) = a_i(\lambda) \cup \delta_i(\lambda)$ includes the diagonal box);
    \item the (depth-$M$) triple slicing, defined later in Definition \ref{def: deep triple slicing} and represented in Figure \ref{fig:triple slicing};
    \item the first-row slicing as in Figure \ref{fig:lambdadown1slicing}, corresponding to the partition $P = \ag \lambda_1, \lambda_{\geq 2} \ad$ of $\lambda$. 
\end{itemize}

\begin{figure}[!ht]
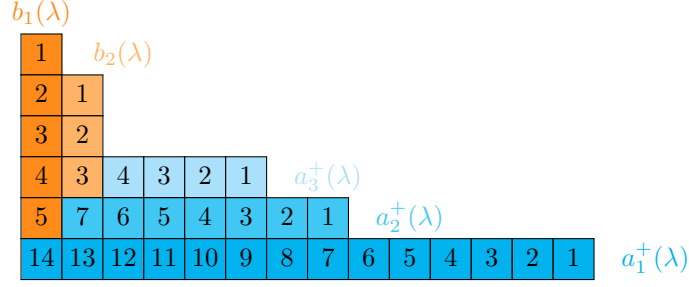

    \begin{center}
    \begin{ytableau}
        \none &\none[\textcolor{orange!90}{b_1(\lambda)}] \\
        \none &*(orange!90)1 &\none &\none[\textcolor{orange!60}{b_2(\lambda)}]\\
        \none &*(orange!90)2&*(orange!60)1\\
        \none &*(orange!90)3&*(orange!60)2\\
        \none &*(orange!90)4&*(orange!60)3&*(cyan!30)4&*(cyan!30)3&*(cyan!30)2&*(cyan!30)1 & \none &\none[\textcolor{cyan!30}{a_3^+(\lambda)}] \\
        \none &*(orange!90)5&*(cyan!60)7&*(cyan!60)6&*(cyan!60)5&*(cyan!60)4&*(cyan!60)3&*(cyan!60)2&*(cyan!60)1 & \none &\none[\textcolor{cyan!60}{a_2^+(\lambda)}]\\
        \none &*(cyan!90)14&*(cyan!90)13&*(cyan!90)12&*(cyan!90)11&*(cyan!90)10&*(cyan!90)9&*(cyan!90)8&*(cyan!90)7&*(cyan!90)6&*(cyan!90)5&*(cyan!90)4&*(cyan!90)3&*(cyan!90)2&*(cyan!90)1 & \none &\none[\textcolor{cyan!90}{a_1^+(\lambda)}] 
    \end{ytableau}
\end{center}

\caption{The stairs slicing for $\lambda = [14,8,6,2,2,1]$. Numbers in the boxes correspond to the hook lengths after slicing.}
\label{fig:stairs slicing}
\end{figure}

In \cite[Proposition 3.7]{TeyssierThévenin2025virtualdegreeswitten}, we proved precise bounds on slicings. 
We now rewrite these results as bounds on ratios of hook products, with a minor rewriting of the upper bounds.

\begin{lemma}\label{lem: slicing from ALS}
    Let $\mu\subset\lambda$ be (non-empty) Young diagrams.
\begin{enumerate}
    \item We have 
    \begin{equation}
        \frac{H(\lambda, \mu_1)}{H(\lambda_1, \mu_1)} \leq e^{3\sqrt{|\lambda|}}.
    \end{equation}
     \item Assume that $r:= |\lambda_{\geq 2}|\leq |\lambda|/4$. Then
    \begin{equation}
        \frac{H(\lambda, \mu_1)}{H(\lambda_1, \mu_1)} \leq e^{2r/|\lambda|}.
    \end{equation}
\end{enumerate}
\end{lemma}

\begin{proof}

Using \cite[Lemma 3.4 (ii) and Proposition 3.7 (d)]{TeyssierThévenin2025virtualdegreeswitten}, we have
\begin{equation}
\frac{H(\lambda, \mu_1)}{H(\lambda_1, \mu_1)} \leq \frac{H(\lambda, \lambda_1)}{H(\lambda_1, \lambda_1)} \leq e^{3\sqrt{r}} \leq e^{3\sqrt{|\lambda|}},
\end{equation}
which proves (a). 
On the other hand, under the assumption $r\leq |\lambda|/4$ we have $\lambda_1 - r + 1  = |\lambda| - 2 r + 1 \geq |\lambda|/2$. Therefore, by \cite[Lemma 3.4 (ii) and Proposition 3.7 (c)]{TeyssierThévenin2025virtualdegreeswitten}, we have
\begin{equation}
    \frac{H(\lambda, \mu_1)}{H(\lambda_1, \mu_1)} \leq \frac{H(\lambda, \lambda_1)}{H(\lambda_1, \lambda_1)} \leq 1+\frac{r}{\lambda_1-r+1}
    \leq 1+ 2 \frac{r}{|\lambda|}
    \leq e^{2r/|\lambda|}.
\qedhere
\end{equation}
\end{proof}

\subsection{Bounds on characters of fixed-point-free permutations}
Larsen and Shalev \cite{LarsenShalev2008} introduced the notions of virtual degree $D(\lambda)$,
recalled previously in \eqref{eq: def virtual degree},
of a diagram $\lambda$ and orbit growth exponents $E(\sigma)$ of a permutation $\sigma$, and derived precise bounds on the characters $|\ch^\lambda(\sigma)|$ as a function of these quantities (for permutations $\sigma$ that have few fixed points and all diagrams $\lambda$). 
The quantity $E(\sigma)$ is defined as follows.

\begin{definition}
    Let $n\geq 1$ and $\sigma\in\mathfrak{S}_n$. Denote by $f_i(\sigma)$ the number of cycles of length $i\geq 1$ of $\sigma$, and for $k\geq 1$, set $\Sigma_k = \Sigma_k(\sigma):= \sum_{1\leq i\leq k} if_i(\sigma)$. We define the orbit growth sequence $(e_i)_{i \geq 1}$ by $n^{e_1+\ldots+e_i}=\max\pg \Sigma_i, 1 \pd$ for all $i \geq 1$, and set
\begin{equation}
\label{eq:orbit growth exponent}
    E(\sigma) = \sum_{i\geq 1} \frac{e_i}{i}.
\end{equation}
\end{definition}

In \cite{TeyssierThévenin2025virtualdegreeswitten}, we improved the approximation on virtual degrees of \cite[Theorem 2.1]{LarsenShalev2008} from $D(\lambda) = d_{\lambda}^{1+o(1)}$ to $D(\lambda) = d_{\lambda}^{1+O(1/\ln n)}$. This led to the following improved bound on characters of fixed-point-free permutations.

\begin{lemma}[{\cite[Theorem 1.6]{TeyssierThévenin2025virtualdegreeswitten}}]\label{lem:ALS improved character bound}
    There exists a universal constant $\Cvirt$ such that for every $n\geq 2$, any permutation $\sigma \in \mathfrak{S}_n$ and any integer partition $\lambda\vdash n$, we have
    \begin{equation}
        \bg \ch^\lambda(\sigma)\bd \leq d_\lambda^{\beta(\sigma)},
    \quad\text{where}\quad
        \beta(\sigma) = \pg 1+ \frac{\Cvirt}{\ln n} \pd E(\sigma).
    \end{equation}
In particular, when $\sigma$ is fixed-point-free, we have $E(\sigma) \leq \frac{1}{2}$ and therefore 
\begin{equation}
        \bg \ch^\lambda(\sigma)\bd \leq d_\lambda^{\frac{1}{2}\pg 1 + \frac{\Cvirt}{\ln n} \pd}.
    \end{equation}
\end{lemma}

\subsection{Some classical bounds}

We collect some elementary/classical lemmas that we will use throughout the article.
\begin{lemma}
\label{lem:dlambda plus petit que racine de factorielle n}
Let $\lambda \vdash n$. Then, $d_\lambda \leq \sqrt{n!}$
\end{lemma}
\begin{proof}
     We have the identity $\sum_{\nu \vdash n} d_\nu^2 = n!$ by the Fourier isomorphism (see \cite[Section 1.3] {LivreMéliot2017RepresentationTheoryofSymmetricGroups}). Therefore $d_\lambda \leq \sqrt{\sum_{\nu \vdash n} d_\nu^2} = \sqrt{n!}$. 
\end{proof}

The following lemma is a classical counting argument, see for instance \cite[Fact 3.D.1 on Pages 39-40]{LivreDiaconis1988}. We give a version with multiple diagrams and a computational proof.
\begin{lemma}\label{lem: binomial bound number standard tableaux fragmented}
Let $\lambda$ be a Young diagram and let $p\geq 2$ be an integer. Write $\lambda$ as a disjoint union $\lambda(1)\cup \cdots \cup \lambda(p)$, such that for each $1\leq j \leq p$, $\cup_{1\leq i\leq j}\lambda(i)$ is a Young diagram. Then
\begin{equation}
    \frac{d_\lambda}{|\lambda|!} \leq \prod_{i=1}^p \frac{d_{\lambda(i)}}{|\lambda(i)|!}.
\end{equation}
\end{lemma}
\begin{proof}
    First assume that $p=2$. Then $\lambda\backslash\lambda(1) = \lambda(2)$. By the hook length formula, the Naruse hook length formula \eqref{e:NHLF rewritten}, and since $S(\lambda, \lambda(1)) \geq H(\lambda, \lambda(1)) \geq H({\lambda(1)}, {\lambda(1)})$, we have
\begin{equation}
    \frac{d_{\lambda(1)}d_{\lambda(2)}}{d_{\lambda}} \frac{|\lambda|!}{|\lambda(1)|!|\lambda(2)|!} = \frac{d_{\lambda(1)}}{|\lambda(1)|!}\pg \frac{d_{\lambda\backslash\lambda(1)}}{d_{\lambda}} |\lambda|^{\downarrow |\lambda(1)|}\pd = \frac{1}{H(\lambda(1),\lambda(1))} S(\lambda, \lambda(1)) \geq 1,
\end{equation}
which is the desired result for $p=2$. It follows by induction that the result holds for any $p\geq 2$.
\end{proof}

We also recall classical bounds on factorials and binomial coefficients.
\begin{lemma}\label{lem: comparaison puissances et puissances factorielles}
\begin{enumerate}
    \item Let $k\geq 0$ be an integer. Then $k! \geq (k/e)^k$.
    \item Let $n$ and $k$ be integers with $0 \le k \le n$.
    Then,
    \begin{equation}
        \frac{n^k}{n^{\downarrow k}} \le e^k
    \quad\text{and}\quad
        \binom nk \le \biggl(\frac{en}{k}\biggr)^k.
    \end{equation}
\end{enumerate}
\end{lemma}

\begin{proof}
\begin{enumerate}
    \item 
    This holds for $k=0$, and for $k\geq 1$ follows from rearranging the sum--integral inequality
    \begin{equation}
        \ln(k!)
    =
        \sum_{r=1}^k
        \ln r
    \ge
        \int_1^k
        \ln x
        \, dx
    =
        k \ln k - k + 1
    \ge
        k \ln k - k.
    \end{equation}

    \item 
    By (a),
    \(
        \binom nk
    =
        n^{\downarrow k} / k!
    \le
        n^k / (k/e)^k
    =
        (en/k)^k.
    \)
    Recall that
    \(
        \binom nk
    =
        n^{\downarrow k}/k^{\downarrow k}
    \geq
        (n/k)^k.
    \)
    So,
    \begin{equation}
        \frac{n^k}{n^{\downarrow k}}
        \bigg/
        \frac{k^k}{k^{\downarrow k}}
    =
        \biggl(\frac nk\biggr)^k
        \bigg/
        \binom nk
    \le
        1,
    \quad\text{and hence}\quad
        \frac{n^k}{n^{\downarrow k}}
    \le
        \frac{k^k}{k!}
    \le
        e^k.
    \qedhere
    \end{equation}
\end{enumerate}
\end{proof}

For $n\geq 1$ let $p(n)$ be the number of integer partitions of the integer $n$.
\begin{lemma}\label{lem:def Cpart}
    There exists a constant $\Cpart$ such that $p(n) \leq e^{\Cpart\sqrt{n}}$ for every $n\geq 0$.
\end{lemma}
\begin{proof}
    This follows immediately from the asymptotic equivalent $p(n) \sim e^{\pi\sqrt{2n/3}}/(4n\sqrt{3})$.
\end{proof}
We recall that the conjugate diagram of a diagram $\lambda$ is denoted by $\lambda'$, and the sign of a permutation $\sigma$ is denoted by $\sgn(\sigma)$.

\begin{lemma}\label{lem: identity characters sign transpose conjugation}
    For any $n\geq 1$, $\sigma\in \kS_n$, and $\lambda \vdash n$, we have
\begin{equation}
    \ch^{\lambda'}(\sigma) = \sgn(\sigma)\ch^\lambda(\sigma).
\end{equation}
\end{lemma}
\begin{proof}
    This follows from the Murnaghan--Nakayama rule (\cite[Theorem 3.10]{LivreMéliot2017RepresentationTheoryofSymmetricGroups}), since transposing a ribbon changes the parity of its height if and only if the ribbon has an even number of boxes. 
\end{proof}
\subsection{Bounds on dimensions of representations}
The next two lemmas give useful bounds on $d_\lambda$.
\begin{lemma}
\label{lem:d lambda plus petit que 6 delta n}
    Let $\lambda$ be a Young diagram. Denote by $\delta = \delta(\lambda)$ the length of its diagonal. Then
    \begin{equation}
        d_\lambda \leq (2\delta)^n.
    \end{equation}
\end{lemma}
\begin{proof}
   The dimension $d_\lambda$ is the number of standard tableaux of $\lambda$, that is, the number of ways we can remove corners one by one from $\lambda$ until we get an empty diagram. At each step there are at most $2\delta$ corners, so overall we have $d_\lambda \leq (2\delta)^n$.
\end{proof}

\begin{lemma}\label{lem: bound dimension M N}
    Let $n,M \geq 1$ be integers, and let $\lambda \vdash n$.
    Denote $m = \bg \lambda^{>M}\bd$. Then
    \begin{equation}
        d_\lambda \leq (4M)^n m^{m/2}.
    \end{equation}
\end{lemma}
\begin{proof}
We have $\binom{n}{m} \leq 2^n$, $d_{\lambda^{\leq M}} \leq (2M)^{n-m}$ by Lemma \ref{lem:d lambda plus petit que 6 delta n}, and $d_{\lambda^{>M}} \leq \sqrt{m!} \leq {m}^{m/2}$, so
\begin{equation}
    d_\lambda \leq \binom{n}{m} d_{\lambda^{\leq M}} d_{\lambda^{>M}} \leq 2^n (2M)^{n-m} {m}^{m/2} = (4M)^n \pg \frac{m}{4M^2} \pd^{m/2} \leq (4M)^n m^{m/2}. \qedhere
\end{equation}
\end{proof}

\subsection{An iterated branching rule}\label{s:An iterated branching rule}

There are many identities for characters of symmetric groups. In the language of the Murnaghan--Nakayama rule (see \cite[Section 3.1]{LivreMéliot2017RepresentationTheoryofSymmetricGroups}), Larsen and Shalev \cite{LarsenShalev2008} peeled the ribbons in decreasing order of length, which corresponds to writing the cycle structure in decreasing order of the lengths of cycles. What we do is that we peel first the ribbons of size 1 and then the other ribbons. This peeling order is equivalent to the iterated branching rule \eqref{eq:branching rule intro non renormalisée}. 

Let us define the terms in this branching rule in more details.
If $\sigma$ is a permutation of $\kS_n$ with support $k$ (that is, $n-k$ fixed points), we define the permutation $\sigma^* \in \kS_k$ as follows: let $\ell_1 \geq \ell_2 \geq \ldots \geq \ell_N$ be the ordered lengths of the cycles of length $\geq 2$ of $\sigma$. Then we define 
\begin{equation}\label{eq:définition sigma étoile}
    \sigma^* = (1 \ldots \ell_1) (\ell_1+1 \ldots \ell_1+\ell_2) \cdots (k-\ell_N+1 \, \ldots \, k).
\end{equation}
In particular, $\sigma^*$ has the same number of $i$-cycles as $\sigma$ for each $i\geq 2$, and $\sigma^*$ has no fixed point. For example if $\sigma = ( 1\; 3 \; 5 \; 8) (7 \; 2) \in \kS_9$, then $\sigma^* = (1 \; 2 \; 3 \; 4)(5 \; 6) \in \kS_6$.

If $\lambda$ and $\mu$ are two Young diagrams such that $\mu \subset \lambda$, we recall that we denote by $d_{\lambda \backslash \mu}$ the number of standard tableaux of the skew diagram $\lambda \backslash \mu$. We finally recall that we write $\mu \subset_{\vdash k} \lambda$ to emphasize that $\mu$ is a sub Young diagram of $\lambda$ with $k$ boxes. We use the the following renormalized form of the iterated branching rule: for $\lambda \vdash n$ and $\sigma \in \kS_n$, we have
\begin{equation}\label{eq:branching rule}
        \chi^\lambda(\sigma) = \sum_{\mu \subset_{\vdash k} \lambda} \ch^\mu(\sigma^*) \frac{d_{\lambda\backslash \mu}}{d_\lambda}.
\end{equation}

\subsection{Combining branching rules with ideas of Larsen--Shalev and Naruse}\label{s: Combining branching rules with ideas of Larsen--Shalev and Naruse}
We start with the following classical counting lemma.
\begin{lemma}\label{lem: elementary bound d mu times d lambda backslash mu less than d lambda}
    Let $\lambda, \mu$ be Young diagrams and $n \geq 1$ such that $\lambda \vdash n$ and $\mu \subset \lambda$. Then
    \begin{equation}
        d_{\mu} d_{\lambda\backslash \mu} \leq d_{\lambda}.
    \end{equation}
In particular for any $\alpha \geq 0$ we have
\begin{equation}
 \pg d_\mu\pd ^{\alpha}\frac{d_{\lambda\backslash \mu}}{d_\lambda} \leq \pg \frac{d_{\lambda\backslash \mu}}{d_\lambda} \pd^{1-\alpha}.
\end{equation}
\end{lemma}

\begin{proof}
    By definition, $d_{\mu} d_{\lambda\backslash \mu}$ counts the number of standard tableaux of $\lambda$ with the numbers from 1 to $k$ in $\mu$ and the numbers from $k+1$ to $n$ in $\lambda\backslash \mu$, and is therefore at most the total number $d_\lambda$ of standard tableaux of $\lambda$. Thus, $d_\mu d_{\lambda \backslash \mu} \leq d_\lambda$, or equivalently $\frac{d_\mu d_{\lambda \backslash \mu}}{d_\lambda} \leq 1$. For $\alpha \geq 0$, we therefore have
    \begin{equation}
        \pg d_\mu\pd ^{\alpha}\frac{d_{\lambda\backslash \mu}}{d_\lambda} = \pg d_\mu\frac{d_{\lambda\backslash \mu}}{d_\lambda} \pd ^{\alpha} \pg\frac{d_{\lambda\backslash \mu}}{d_\lambda} \pd^{1-\alpha} \leq \pg\frac{d_{\lambda\backslash \mu}}{d_\lambda} \pd^{1-\alpha}. \qedhere
    \end{equation}
\end{proof}

Recall from Lemma \ref{lem:ALS improved character bound} that $\beta(\sigma) = \pg 1+ \frac{\Cvirt}{\ln n} \pd E(\sigma)$ for $\sigma \in \kS_n$, and set $\theta_k = \frac{1}{2} \pg 1+ \frac{\Cvirt}{\ln k} \pd$ for $k\geq 2$.
\begin{lemma}\label{lem: improved LS inside MN rewritten}
    Let $n\geq 2$, $2\leq k \leq n$, $\sigma \in \kS_n$ with support size $k$, and $\lambda \vdash n$. Then
\begin{equation}\label{eq: MN rewritten + ALS}
    \bg\chi^\lambda(\sigma)\bd\leq \sum_{\mu \subset_{\vdash k} \lambda} d_\mu^{\beta(\sigma^*)}\frac{d_{\lambda\backslash \mu}}{d_\lambda} \leq \sum_{\mu \subset_{\vdash k} \lambda} d_\mu^{\theta_k}\frac{d_{\lambda\backslash \mu}}{d_\lambda}.
\end{equation}
\end{lemma}
\begin{proof}
    Applying the triangle inequality in \eqref{eq:branching rule}, we get
    \begin{equation}
\bg\chi^\lambda(\sigma)\bd = \bg\sum_{\mu \subset_{\vdash k} \lambda} \ch^\mu(\sigma^*) \frac{d_{\lambda\backslash \mu}}{d_\lambda}\bd \leq \sum_{\mu \subset_{\vdash k} \lambda} \bg  \ch^\mu(\sigma^*) \bd \frac{d_{\lambda\backslash \mu}}{d_\lambda}.
    \end{equation}
    The first inequality follows from Lemma \ref{lem:ALS improved character bound}, applied to each $\mu \subset_{\vdash k} \lambda$. 
    Note that $\sigma^* \in \kS_k$, so $\beta(\sigma^*) = \pg 1+ \frac{\Cvirt}{\ln k} \pd E(\sigma^*)$. Moreover $\sigma^*$ is a fixed-point-free permutation, so we have $E(\sigma^*) \leq 1/2$, which proves the second inequality since $\theta_k = \frac{1}{2} \pg 1+ \frac{\Cvirt}{\ln k} \pd$.
\end{proof}

\begin{lemma}\label{lem: bound on MN rewritten with ALS renomalized and simplified lemme fusionné}
    Let $1/2< \theta \leq 1$. Set $k_0(\theta) = \min \ag k'\geq 2 \du \frac{1}{2} \pg 1+ \frac{\Cvirt}{\ln k'} \pd \leq \theta\ad$. Let $n$ and $k$ be integers such that $k_0(\theta)\leq k \leq n$. Let $\sigma \in \kS_n$ with support size $k$, and let $\lambda \vdash n$. Then
\begin{equation}
    \bg \chi^\lambda(\sigma) \bd \leq e^{\Cpart \sqrt{k}} \max_{\mu \subset_{\vdash k} \lambda} d_\mu^{\theta} \frac{d_{\lambda\backslash \mu}}{d_\lambda}.
\end{equation}
\end{lemma}
\begin{proof} 
Recall from Lemma \ref{lem:def Cpart} that $p(k)\leq e^{\Cpart\sqrt{k}}$. Moreover, since $k\geq k_0(\theta)$ we have $\theta_k \leq \theta$. Plugging this into Lemma \ref{lem: improved LS inside MN rewritten} therefore gives
\begin{equation}
    \bg\chi^\lambda(\sigma)\bd \leq \sum_{\mu \subset_{\vdash k} \lambda} d_\mu^{\theta_k}\frac{d_{\lambda\backslash \mu}}{d_\lambda} \leq p(k) \max_{\mu \subset_{\vdash k} \lambda} d_\mu^{\theta_k}\frac{d_{\lambda\backslash \mu}}{d_\lambda} \leq e^{\Cpart\sqrt{k}} \max_{\mu \subset_{\vdash k} \lambda} d_\mu^{\theta}\frac{d_{\lambda\backslash \mu}}{d_\lambda}. \qedhere
\end{equation}
\end{proof}

\section{Bounds for high-level representations}\label{s: bounds for high-level representations}

\subsection{Bounds on the number of excited diagrams and excited sums}
To bound characters, we need to understand better the ratios $\frac{d_{\lambda\backslash \mu}}{d_\lambda} = \frac{S(\lambda,\mu)}{|\lambda|^{\downarrow |\mu|}}$. Asymptotics for $d_{\lambda\backslash \mu}$ were proved in different regimes for diagrams $\lambda$ and $\mu$ with short hook lengths; see \cite{DousseFéray2019skewdiagramscharacters, MoralesPakTassy2022skewshapelozenge}.
Here we need bounds that combine well with deep triple slicings (see Figure \ref{fig:triple slicing} below). In order to do so we have to prove bounds on the excited sums 
\begin{equation}
    S(\lambda, \mu) = \sum_{E \in \dgex(\lambda, \mu)} H(\lambda, E).
\end{equation}
We start with the following standard bound.

\begin{lemma}\label{lem:bound number superexcited inner diagrams brutal product}
    Let $\lambda, \mu $ be Young diagrams such that $\mu \subset \lambda$. Denote $n = |\lambda|$. Then
\begin{equation}\label{eq:bound S lambda mu via number superexcited diagrams}
  S(\lambda, \mu) \leq  2^n H(\lambda, \mu).
\end{equation}
\end{lemma}
\begin{proof}
Since each excited diagram in $\cE(\lambda, \mu)$ is a subset of the $n$ boxes of $\lambda$, we have $|\cE(\lambda, \mu)| \leq 2^n$. Moreover, $\max_{E \in \cE(\lambda, \mu)} H(\lambda, E) = H(\lambda, \mu)$. Therefore $S(\lambda, \mu) \leq |\cE(\lambda, \mu)| \max_{E \in \cE(\lambda, \mu)} H(\lambda, E) \leq 2^n H(\lambda, \mu)$.
\end{proof}

However this bound is too weak in general, and we need more precise bound on $|\cE(\lambda, \mu)|$ and $S(\lambda, \mu)$ which depend on the shapes of $\lambda$ and $\mu$. Let us first bound $|\cE(\lambda, \mu)|$ when $\mu$ consists of a single row.

\begin{lemma}\label{lem:bound number superexcited diagrams inner flat coeff binom}
    Let $\lambda$ be a Young diagram, $\delta$ be the length of the diagonal of $\lambda$, and let $\ell \leq \lambda_1$. Then
        \begin{equation}
    |\cE(\lambda,[\ell])| \leq \min \pg \binom{\ell + \lf n/\ell \rf}{\ell}, \binom{\ell + \delta}{\ell} \pd .
\end{equation}
\end{lemma}
\begin{proof}
Set $n = |\lambda|$. A diagram $E \in \cE(\lambda,[\ell])$ is totally determined by the horizontal coordinates of its boxes, which are distinct and are between $1$ and $\ell + \lf n/\ell \rf$. 
To see this, observe first that the maximum horizontal coordinate of a box in an excited diagram $E \in \cE(\lambda,[\ell])$ is reached by exciting the box $(1,\ell)$. Furthermore, if the box $(1,\ell)$ can be excited $j$ times, then necessarily the box $(1+j,\ell+j)$ is in $\lambda$, and in particular $j\ell \leq n$. Hence, the maximum horizontal coordinate of a box in $E$ is at most $\ell + \lf n/\ell \rf$.
Therefore, the $\ell$ boxes must all have distinct horizontal coordinates, each in $\{1, \ldots, \floor{n/\ell}\}$, and hence
\begin{equation}
    |\cE(\lambda,[\ell])|
\leq
    \binom{\ell + \lf n/\ell \rf}{\ell}.
\end{equation}
Similarly, each box cannot be moved more than $\delta-1 \le \delta$ times.
So, $|\cE(\lambda, [\ell])|$ is at most the number of weakly increasing integer-valued sequences with values in $[0, \delta]$:
\begin{equation}
        |\cE(\lambda,[\ell])| \leq \binom{\ell + \delta}{\ell}
    \end{equation}
by the usual stars and bars argument, which concludes the proof.
\end{proof}

\begin{lemma}\label{lem:bound number superexcited diagrams inner flat square root n}
    Let $\lambda$ be a Young diagram, and let $1\leq \ell \leq \lambda_1$. Then
        \begin{equation}
    |\cE(\lambda,[\ell])| \leq e^{5\sqrt{|\lambda|}}.
\end{equation}
\end{lemma}
\begin{proof}
    Denote $n=|\lambda|$, and by $\delta = \delta(\lambda)$ the diagonal length of $\lambda$. Since the square $[\delta^\delta]$ is a sub-diagram of $\lambda$, we have $\delta \le \sqrt n$.
    By Lemma \ref{lem:bound number superexcited diagrams inner flat coeff binom}, if $\ell \leq \sqrt{n}$ we have 
\begin{equation}
    |\cE(\lambda,[\ell])| \leq \binom{\ell + \delta}{\ell} \leq 2^{\ell + \delta} \leq 2^{\sqrt{n} + \sqrt{n}} = 2^{2\sqrt{n}} \leq e^{5\sqrt{|\lambda|}}.
\end{equation}
On the other hand, again by Lemma \ref{lem:bound number superexcited diagrams inner flat coeff binom}, if $\ell > \sqrt{n}$, using $\binom{b}{a} \leq (eb/a)^a$ for all $a,b > 0$, we get
\begin{equation}
    |\cE(\lambda,[\ell])| \leq \binom{\ell + \lf n/\ell \rf}{\ell} =  \binom{\ell + \lf n/\ell \rf}{\lf n/\ell \rf} \leq \binom{2\ell}{\lf n/\ell \rf} \leq \binom{2\ell}{\lc n/\ell \rc} \leq \pg \frac{2e\ell}{\lc n/\ell \rc} \pd^{\lc n/\ell \rc},
\end{equation}
so 
\begin{equation}
    |\cE(\lambda,[\ell])| \leq \pg \frac{2e\ell}{n/\ell} \pd^{\lc n/\ell \rc} \leq \pg \frac{2e\ell^2}{n} \pd^{2 n/\ell}.
\end{equation}
Set $y = \sqrt{n}/\ell$. By assumption we have $0\leq y\leq 1$ so $y^y \geq e^{-1/e}$. We conclude that
\begin{equation}
    |\cE(\lambda,[\ell])|^{1/\sqrt{n}} \leq \pg \frac{2e}{y^2} \pd^{2 y} \leq (2e)^2 y^{-4y} \leq (2e)^2 e^{4/e} \leq e^5. \qedhere
\end{equation}
\end{proof}

\begin{lemma}\label{lem: bound on the number of excited diagrams and excited sums when shifting lines}
    Let $\lambda$ be a Young diagram. Let $\ell\geq 1$ and $i,j\geq 0$ be integers. Let $L$ be either $[\ell]$ or $[1^\ell]$ and $\Tilde{L} = \Shift_{i,j}(L)$. Assume that $\Tilde{L} \subset \lambda$. Then 
   \begin{equation}
    |\cE(\lambda,\Tilde{L})| \leq |\cE(\lambda,L)| \quad  \text{ and } \quad S(\lambda, \Tilde{L}) \leq S(\lambda, L).
\end{equation}
\end{lemma}
\begin{proof}
    Since we can only excite boxes towards north east, the set $\cE(\lambda,\Tilde{L})$ is in bijection with $\cE(\nu ,L)$, where $\nu = \Shift_{-i, -j}(\lambda \cap [i, \infty) \times [j, \infty))$, and the corresponding hook lengths are the same. The two inequalities then follow, since $\nu \subset \lambda$ by definition. 
\end{proof}

\begin{lemma}\label{lem: bound on the number of excited diagrams and excited sums when fragmenting into row and column diagrams}
    Let $\lambda$ be a Young diagram and $E_0$ be a subset of the boxes of $\lambda$. Write $E_0$ as a disjoint union of $p_R$ row diagrams $R_1,\ldots, R_{p_R}$ and $p_C$ column diagrams $C_1,\ldots,C_{p_C}$, whose lengths we denote by $r_1,\ldots, r_{p_R}, c_1,\ldots,c_{p_C}$. Then
    \begin{equation}
    |\cE(\lambda, E_0)| \leq \pg \prod_{i=1}^{p_R} \bg \cE(\lambda, [r_i])  \bd \pd\pg \prod_{i=1}^{p_C} \bg \cE(\lambda, [1^{c_i}])  \bd \pd,
\end{equation}
and
 \begin{equation}
    S(\lambda, E_0) \leq \pg \prod_{i=1}^{p_R} S(\lambda, [r_i]) \pd\pg \prod_{i=1}^{p_C} S(\lambda, [1^{c_i}]) \pd.
\end{equation}
\end{lemma}
\begin{proof}
    This follows immediately from Lemma \ref{lem: fragmentation bound for excited sums} and Lemma \ref{lem: bound on the number of excited diagrams and excited sums when shifting lines}. 
\end{proof}

\begin{lemma}
    \label{lem:bound number superexcited inner diagrams multiline product}
    Let $\lambda$ be a Young diagram and $p\geq 1$ be an integer. Let $
    E_0 \subset \lambda$ be a subset of the boxes of $\lambda$ that can be written as a disjoint union of $p$ row and column diagrams. Then
\begin{equation}
    |\cE(\lambda,E_0)| \leq e^{5 p\sqrt{|\lambda|}},
\end{equation}
and in particular
\begin{equation}\label{eq:bound S lambda mu via number superexcited diagrams union of rows}
    S(\lambda, E_0) \leq e^{5 p\sqrt{|\lambda}|} H(\lambda, E_0).
\end{equation}
\end{lemma}
\begin{proof}

Write $E_0$ as a disjoint union of row and column diagrams. For each of these row or column diagrams $L$, by Lemmas \ref{lem:bound number superexcited diagrams inner flat square root n} and \ref{lem: bound on the number of excited diagrams and excited sums when shifting lines} we have $|\cE(\lambda,L)| \leq e^{5\sqrt{|\lambda|}}$. We deduce from Lemma \ref{lem: bound on the number of excited diagrams and excited sums when fragmenting into row and column diagrams} that $|\cE(\lambda,E_0)| \leq e^{5 p\sqrt{|\lambda|}}$.
The second point follows from bounding the sum $S(\lambda, E_0)~=~\sum_{E\in \cE(\lambda, E_0)} H(\lambda, E)$ by its maximal term $H(\lambda, E_0)$, times its number of terms $|\cE(\lambda, E_0)|$.
\end{proof}

\subsection{Deep triple slicings of diagrams}

Diaconis and Shahshahani \cite{DiaconisShahshahani1981} decomposed diagrams according to how many boxes they have on the first row. While this works for transpositions, for other conjugacy classes it works in general only for low-level representations, that is, when the first row is very long. For exponential character bounds, or equivalently to bound mixing times in $L^2$ distance for the simple random walk on $\Cay(\kS_n, \cC)$, the goal is to bound for each $\lambda$ the quantity 
\begin{equation}
    d_\lambda \bg \chi^\lambda(\cC)\bd^t
\end{equation}
for $t$ close to the mixing time. The key observation is that dimensions of diagrams depend almost only on the number of boxes deep inside the diagram, while characters depend almost only on the geometry of the first rows and columns. This motivates the following decomposition.

\begin{definition}\label{def: deep triple slicing}
    Let $\lambda$ be a Young diagram, and let $M\geq 1$ be an integer. The \textit{depth-$M$ triple slicing (or decomposition)} of $\lambda$ is the triplet of diagrams $(\lambda(1), \lambda(2), \lambda(3)) = (\lambda_1, \lambda^{\leq M} \backslash \lambda_1, \lambda^{>M})$. We illustrate this in Figure \ref{fig:triple slicing}.
\end{definition}

\begin{figure}[!ht]
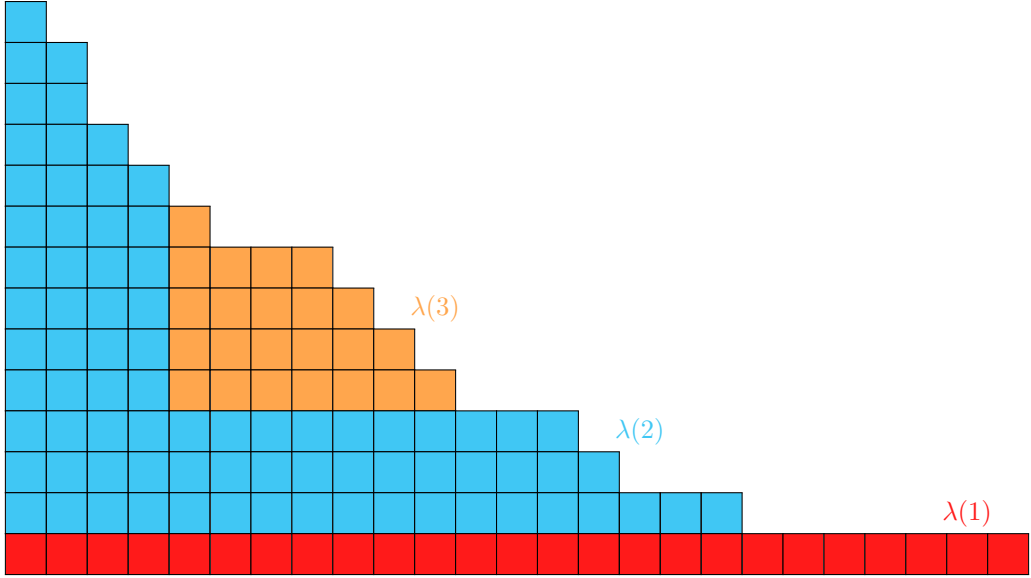

    \begin{center}
    \begin{ytableau}
        \none &*(cyan!60)  \\
        \none&*(cyan!60)&*(cyan!60) \\
        \none&*(cyan!60)&*(cyan!60) \\
        \none&*(cyan!60)&*(cyan!60)&*(cyan!60) \\
        \none&*(cyan!60)&*(cyan!60)&*(cyan!60)&*(cyan!60) \\
        \none&*(cyan!60)&*(cyan!60)&*(cyan!60)&*(cyan!60)&*(orange!70) \\
        \none &*(cyan!60)&*(cyan!60)&*(cyan!60)&*(cyan!60)&*(orange!70)&*(orange!70)&*(orange!70)&*(orange!70)\\
        \none &*(cyan!60)&*(cyan!60)&*(cyan!60)&*(cyan!60)&*(orange!70)&*(orange!70)&*(orange!70)&*(orange!70)&*(orange!70)&\none &\none[\textcolor{orange!70}{\lambda(3)}]\\
        \none &*(cyan!60)&*(cyan!60)&*(cyan!60)&*(cyan!60) &*(orange!70)&*(orange!70)&*(orange!70)&*(orange!70)&*(orange!70)&*(orange!70)\\
        \none &*(cyan!60)&*(cyan!60)&*(cyan!60)&*(cyan!60)&*(orange!70)&*(orange!70)&*(orange!70)&*(orange!70)&*(orange!70)&*(orange!70)&*(orange!70)&\none \\
        \none &*(cyan!60)&*(cyan!60)&*(cyan!60)&*(cyan!60)&*(cyan!60)&*(cyan!60)&*(cyan!60)&*(cyan!60)&*(cyan!60)&*(cyan!60)&*(cyan!60)&*(cyan!60) &*(cyan!60)&*(cyan!60)&\none &\none[\textcolor{cyan!60}{\lambda(2)}]\\
        \none &*(cyan!60)&*(cyan!60)&*(cyan!60)&*(cyan!60)&*(cyan!60)&*(cyan!60)&*(cyan!60)&*(cyan!60)&*(cyan!60)&*(cyan!60)&*(cyan!60)&*(cyan!60)&*(cyan!60)&*(cyan!60)&*(cyan!60)& \none \\
        \none &*(cyan!60)&*(cyan!60)&*(cyan!60)&*(cyan!60)&*(cyan!60)&*(cyan!60)&*(cyan!60)&*(cyan!60)&*(cyan!60)&*(cyan!60)&*(cyan!60)&*(cyan!60)&*(cyan!60)&*(cyan!60)&*(cyan!60) &*(cyan!60)&*(cyan!60)&*(cyan!60)& \none & \none& \none& \none& \none &\none[\textcolor{red!90}{\lambda(1)}]\\
        \none &*(red!90)&*(red!90)&*(red!90)&*(red!90)&*(red!90)&*(red!90)&*(red!90)&*(red!90)&*(red!90)&*(red!90)&*(red!90)&*(red!90)&*(red!90)&*(red!90)&*(red!90)&*(red!90)&*(red!90)&*(red!90)&*(red!90)&*(red!90)&*(red!90)&*(red!90)&*(red!90)&*(red!90)&*(red!90) & \none  
    \end{ytableau}
\end{center}

\caption{The depth-4 triple slicing for $\lambda = [25,18,15,14,11,10,9,8,,5,4,3,2,2,1]$.}
\label{fig:triple slicing}
\end{figure}

In what follows, given an integer $M\geq 1$, two integers $0\leq k\leq n$, and two Young diagrams $\mu \subset \lambda$, we always write $n_i = |\lambda(i)|$ and $k_i = |\mu(i)|$ for $1\leq i\leq 3$, without specifying the dependence on $M$, and for compactness we may call $(\lambda(i), \mu(i), n_i, k_i)_{1\leq i\leq 3}$ the parameters of the depth-$M$ triple slicing.
Note that by definition we have $n_1 + n_2 + n_3 = n$ and $k_1 + k_2+k_3 = k$.

\subsection{Bounds on biased skew dimensions via deep triple slicings}

Let $\mu \subset \lambda$ be two Young diagrams and let $\theta \in \bbR$. We define the $\theta$-biased skew dimension to be the quantity
    \begin{equation}
     B_\theta(\lambda, \mu) := d_\mu^\theta \frac{d_{\lambda\backslash \mu}}{d_\lambda}  = d_\mu^\theta\frac{S(\lambda, \mu)}{n^{\downarrow k}},
    \end{equation}
where the last equality is due to the Naruse hook length formula (written as \eqref{e:NHLF rewritten}).

In particular, if $0\leq \theta\leq 1$, by Lemma \ref{lem: elementary bound d mu times d lambda backslash mu less than d lambda} we have
\begin{equation}\label{eq: bound on biased skew dimensions by 1}
    B_\theta(\lambda, \mu) \leq \pg \frac{d_{\lambda\backslash \mu}}{d_\lambda} \pd^{1-\theta} \leq 1.
\end{equation}

\begin{proposition}\label{prop: borne dimensions gauches biaisées par découpe profonde triple}
    Let $\theta>0$. Let $n$ and $k$ be integers such that $0\leq k\leq n$. Let $\lambda \vdash n$ and $\mu \vdash k$ such that $\mu \subset \lambda$. Let $M\geq 1$ be an integer, and denote the parameters of the depth-$M$ triple slicings of $\lambda$ and $\mu$ by $(\lambda(i), \mu(i), n_i, k_i)_{1\leq i\leq 3}$. Then
 \begin{equation}
     B_\theta(\lambda, \mu) \leq e^{24M\sqrt{n}} \pg \frac{k!}{\prod_{i=1}^3 k_i!} \pd^\theta \frac{\prod_{i=1}^3 n_i^{\downarrow k_i} }{n^{\downarrow k}} \prod_{i=1}^3 B_\theta(\lambda(i), \mu(i)).
 \end{equation}
\end{proposition}
\begin{proof}
    First, by Lemma \ref{lem: binomial bound number standard tableaux fragmented} we have
\begin{equation}\label{eq: bound triple slicing on the dimension of mu}
    d_\mu \leq \frac{k!}{k_1!k_2!k_3!}d_{\mu(1)}d_{\mu(2)}d_{\mu(3)}.
\end{equation}
Also, by Lemma \ref{lem: fragmentation bound for excited sums} we have
\begin{equation}
    S(\lambda, \mu) \leq S(\lambda, \mu(1))S(\lambda, \mu(2)) S(\lambda, \mu(3)) = S(\lambda, \mu(1))S(\lambda, \mu(2)) S(\lambda(3), \mu(3))
\end{equation}
Let us bound the first two terms above separately. We have
\begin{equation}
    S(\lambda, \mu(1)) = \sum_{E\in \cE(\lambda, \mu(1))}H(\lambda, E) \leq \sum_{E\in \cE(\lambda, \mu(1))}H(\lambda, \mu(1))  = |\cE(\lambda, \mu(1))| H(\lambda, \mu(1)),
\end{equation}
and therefore by Lemmas \ref{lem:bound number superexcited diagrams inner flat square root n} and \ref{lem: slicing from ALS} (a) we have
\begin{equation}
   S(\lambda, \mu(1)) \leq e^{5\sqrt{n}} \pg e^{3\sqrt{n}}  H(\lambda(1), \mu(1))\pd = e^{8\sqrt{n}} H(\lambda(1), \mu(1))\leq e^{8M\sqrt{n}} S(\lambda(1), \mu(1)).
\end{equation}
By Lemma \ref{lem:bound number superexcited inner diagrams multiline product}, that is, applying the same argument to each (of the at most $2M$ by definition) row and column diagram in the stairs decompositions (see Figure \ref{fig:stairs slicing}) of $\lambda(2)$ and $\mu(2)$,  we obtain
\begin{equation}
\begin{split}
     S(\lambda, \mu(2)) \leq |\cE(\lambda, \mu(2))| H(\lambda, \mu(2)) \leq & \pg e^{3\sqrt{n}} e^{5\sqrt{n}}\pd^{2M} H(\lambda(2), \mu(2)) \\
     & \leq e^{16M\sqrt{n}} S(\lambda(2), \mu(2)).
\end{split}
\end{equation}
Combining these bounds, we obtain
\begin{equation}
    S(\lambda, \mu)  \leq e^{24M\sqrt{n}}\prod_{i=1}^3 S(\lambda(i), \mu(i)),
\end{equation}
and deduce from the Naruse hook length formula that 
\begin{equation}
    \frac{d_{\lambda\backslash\mu}}{d_{\lambda}}  = \frac{S(\lambda, \mu)}{n^{\downarrow k}}  \leq e^{24M\sqrt{n}}\frac{\prod_{i=1}^3 n_i^{\downarrow k_i} }{n^{\downarrow k}} \prod_{i=1}^3 \frac{d_{\lambda(i)\backslash\mu(i)}}{d_{\lambda(i)}}.
\end{equation}
Combining this with \eqref{eq: bound triple slicing on the dimension of mu}, we obtain the bound
\begin{equation}
    B_\theta(\lambda, \mu) e^{-24M\sqrt{n}} \leq  \pg \frac{k!}{\prod_{i=1}^3 k_i!} \pd^\theta \frac{\prod_{i=1}^3 n_i^{\downarrow k_i} }{n^{\downarrow k}} \prod_{i=1}^3 d_{\mu(i)}^\theta \frac{d_{\lambda(i)\backslash\mu(i)}}{d_{\lambda(i)}},
\end{equation}
which concludes the proof since $\prod_{i=1}^3 d_{\mu(i)}^\theta \frac{d_{\lambda(i)\backslash\mu(i)}}{d_{\lambda(i)}} = \prod_{i=1}^3 B_\theta(\lambda(i), \mu(i))$.
\end{proof}

\begin{lemma}\label{lem: central boxes have small hook lengths}
    Let $n\geq 1$ and $\lambda\vdash n$. Let $M\geq 1$ be an integer. Let $u$ be a box in $\lambda(3) = \lambda^{>M}$. Then $H(\lambda, u) \leq n/M$.
\end{lemma}
\begin{proof}
    Let us write $H(\lambda, u) = h(u) + v(u) -1$, where $h(u)$ (resp. $v(u)$) is the number of boxes on the right of (resp. above) $u$ in $\lambda$, including $u$. Then $\lambda$ must contain an $h(u)\times M$ (resp. $M\times v(u)$) rectangle below (resp. on the left of) the hook of $u$. Therefore  $M H(\lambda, u) \leq M(h(u) + v(u)) = h(u)M + Mv(u) \leq n$, which concludes the proof.
\end{proof}

\subsection{Specific bounds on biased skew dimensions}
Let us fix some notation for this section. Let $\delta>0$, $\eta>0$, and $\rho>0$. Let $\theta = 2/3$. Let $(M_n)$ be a sequence of integers $\geq 1$ such that $1\lll M_n \lll \ln \ln n$. Let $\lambda$ be a Young diagram such that $\lambda_1' \leq \lambda_1 \leq (1-\rho) |\lambda|$. Let $\mu \subset \lambda$ be a Young diagram such that $\eta |\lambda| \leq |\mu| \leq (1-\delta)|\lambda|$. We set $n=|\lambda|$, $k = |\mu|$ and $f = n-k$.
Finally, denote the parameters of the depth-$M_n$ triple decomposition of $\lambda$ and $\mu$ by $(\lambda(i), \mu(i), n_i, k_i)_{1\leq i\leq 3}$. 

\begin{lemma}\label{lem: bound on bsd leading to representations being negligible if all their hooks are small}
    There exist $\varepsilon_1\in (0,\rho)$ and an integer $N_0$ such that if $n\geq N_0$ and $\lambda_1 \leq \varepsilon_1 n$,
    \begin{equation}
    B_\theta(\lambda, \mu) \leq e^{-5n\ln(1/\delta)}.
\end{equation}
\end{lemma}
\begin{proof}
    By Lemmas \ref{lem:bound number superexcited inner diagrams brutal product}, \ref{lem: central boxes have small hook lengths}, and \ref{lem: comparaison puissances et puissances factorielles} (b), we have, using that the maximal hook length in $\lambda$ is at most $2\lambda_1$,
\begin{equation}
    \frac{S(\lambda, \mu)}{n^{\downarrow k}} \leq \frac{2^n H(\lambda, \mu)}{n^{\downarrow k}} \leq \frac{2^n (2\lambda_1)^k}{(n/e)^k} \leq 12^n \pg \frac{\lambda_1}{n} \pd^{\eta n}.
\end{equation}
and therefore, using that $k\gtrsim n$ and recalling \eqref{eq: bound on biased skew dimensions by 1}, 
\begin{equation}
    B_\theta(\lambda, \mu) \leq \pg \frac{S(\lambda, \mu)}{n^{\downarrow k}} \pd^{1/3} \leq \pg O(1) \frac{\lambda_1}{n}\pd^{\Omega_+(n)},
\end{equation}
which concludes the proof.
\end{proof}

Let $\varepsilon_1$ and $N_0$ as in the previous lemma. The following lemma is intended to be applied to low dimensional diagrams $\lambda$, that is in this context, diagrams whose deep center $\lambda^{>M_n}$ has size at most a small constant times $n$.

\begin{lemma}\label{lem: exponential bound with alpha intended for diagrams with small center}
    There exist $\alpha>0$ and $N_1 \geq N_0$ such that if $n\geq N_1$ and $\varepsilon_1 n\leq \lambda_1 \leq (1-\varepsilon_1)n$ then we have
   \begin{equation}
    B_\theta(\lambda,\mu) \leq e^{-2\alpha n \ln(1/\delta)}.
\end{equation}
\end{lemma}
\begin{proof}
    
Bounding $B_\theta(\lambda(i), \mu(i))$ by 1 for $i\in \ag 1,2,3\ad$, we obtain by Proposition \ref{prop: borne dimensions gauches biaisées par découpe profonde triple}
\begin{equation}
\begin{split}
    B_\theta(\lambda, \mu)e^{-24M\sqrt{n}} & \leq \pg \frac{k!}{\prod_{i=1}^3 k_i!} \pd^\theta \frac{\prod_{i=1}^3 n_i^{\downarrow k_i} }{n^{\downarrow k}} \\ & = \pg \frac{\prod_{i=1}^3 \binom{n_i}{k_i}}{\binom{n}{k}} \pd^\theta \pg \frac{\prod_{i=1}^3 n_i^{\downarrow k_i} }{n^{\downarrow k}} \pd^{1 -\theta} \leq \pg \frac{\prod_{i=1}^3 n_i^{\downarrow k_i} }{n^{\downarrow k}} \pd^{1 -\theta}.
\end{split}
\end{equation}
Moreover, $n_2^{\downarrow k_2}n_3^{\downarrow k_3} \leq (n_2 + n_3)^{\downarrow k_2 + k_3} = (n-n_1)^{\downarrow k-k_1}$, so
\begin{equation}
    \frac{\prod_{i=1}^3 n_i^{\downarrow k_i} }{n^{\downarrow k}} \leq \frac{n_1^{\downarrow k_1}(n-n_1)^{\downarrow k-k_1}}{n^{\downarrow k}},
\end{equation}
which is $e^{-\Omega_+(n)}$ since $\varepsilon_1 n\leq \lambda_1 \leq (1-\varepsilon_1)n$ and $k = \Omega(n)$.
Finally, since $e^{-24M_n\sqrt{n}} = e^{o(n)}$ and $\theta<1$, we deduce that $B_\theta(\lambda, \mu) = e^{-\Omega_+(n)}$, which concludes the proof.
\end{proof}

Let $N_1$ as in the previous lemma.

\begin{lemma}\label{lem: f over n bound on biased dimensions for diagrams with large center}
     Let $\varepsilon'>0$. There exists $N_2\geq N_1$ such that if $n\geq N_2$, $\varepsilon_1 n\leq \lambda_1 \leq (1-\varepsilon_1)n$, and $n_3\geq \varepsilon' n$, then
     \begin{equation}
         B_\theta(\lambda, \mu) \leq \pg \frac{f}{n}\pd^{n_3}.
     \end{equation}
\end{lemma}
\begin{proof}
First assume that $k_3 \leq n/\sqrt{\ln M_n}$. 
By Proposition \ref{prop: borne dimensions gauches biaisées par découpe profonde triple}, upper bounding all three $B_\theta(\lambda(i), \mu(i))$ by 1, and also using that $\theta \leq 1$, we get
\begin{equation}
    B_\theta(\lambda, \mu) \leq  \frac{k!}{k_1!k_2!k_3!}\frac{n_1^{\downarrow k_1}n_2^{\downarrow k_2}n_3^{\downarrow k_3}}{n^{\downarrow k}} = \frac{\binom{n_1}{k_1}\binom{n_2}{k_2}\binom{n_3}{k_3}}{\binom{n}{k}}.
\end{equation}
Since $k_3 = o(n)$ and $n_3\gtrsim n$, we have $\binom{n_3}{k_3} \leq (en_3/k_3)^{k_3} =  e^{o(n)}$ so
\begin{equation}
    B_\theta(\lambda, \mu) \leq  e^{o(n)} \frac{\binom{n_1}{k_1}\binom{n_2}{k_2}}{\binom{n}{k}} \leq e^{o(n)} \frac{\binom{n_1+n_2}{k_1+k_2}}{\binom{n}{k}}=  e^{o(n)} \frac{\binom{n-n_3}{k_1+k_2}}{\binom{n}{k_1 + k_2}} = e^{o(n)} \pg \frac{(n-n_3)^{\downarrow k_1+ k_2 }}{n^{\downarrow k_1+ k_2}}\pd,
\end{equation}
and
\begin{equation}
\begin{split}
     \pg \frac{(n-n_3)^{\downarrow k_1+ k_2 }}{n^{\downarrow k_1+ k_2}}\pd =  \pg \frac{(n-(k_1+k_2))^{\downarrow n_3 }}{n^{\downarrow n_3}}\pd = \pg \frac{(f + k_3)^{\downarrow n_3 }}{n^{\downarrow n_3}}\pd
     &\leq \pg \frac{f + k_3}{n}\pd^{n_3}e^{-\Omega_+(n)} \\
     & = e^{o(n)} \pg \frac{f}{n}\pd^{n_3} e^{-\Omega_+(n)},
\end{split}
\end{equation}
so overall we have $B_\theta(\lambda, \mu) \leq e^{-\Omega_+(n)} \pg f/n\pd^{n_3} \leq \pg f/n\pd^{n_3}$, for $n$ large enough.

Now assume that $k_3 > n/\sqrt{\ln M_n}$. We proceed similarly as in the proof of Lemma \ref{lem: bound on bsd leading to representations being negligible if all their hooks are small}. By Lemmas \ref{lem:bound number superexcited inner diagrams brutal product}, \ref{lem: central boxes have small hook lengths}, and \ref{lem: comparaison puissances et puissances factorielles} (b), and since $M_n \to \infty$ as $n\to \infty$, we have
\begin{equation}
\begin{split}
     \frac{S(\lambda(3), \mu(3))}{n_3^{\downarrow k_3}} \leq \frac{2^{n_3} (n/M_n)^{k_3}}{(n_3/e)^{k_3}} = e^{O(n)} \pg \frac{O(1)}{M_n} \pd^{k_3} & = e^{O(n)}\pg \frac{O(1)}{M_n} \pd^{\Omega_+(n/\sqrt{\ln M_n})} \\
     & = e^{\Omega_+(n\sqrt{ \ln M_n})} = e^{-\omega_+(n)}.
\end{split}
\end{equation}
Since $\theta<1$, we deduce that $B_\theta(\lambda(3), \mu(3)) = e^{-\omega_+(n)}$, and it follows from Proposition \ref{prop: borne dimensions gauches biaisées par découpe profonde triple} (and since $M_n = O(\sqrt{n})$) that
\begin{equation}
    B_\theta(\lambda, \mu) \leq e^{24M_n\sqrt{n}} B_\theta(\lambda(3), \mu(3)) = e^{-\omega_+(n)},
\end{equation}
which, since $f\gtrsim n$ and $n_3 \gtrsim n$, is at most $\pg f/n \pd^{n_3}$ for $n$ large enough. This concludes the proof. 
\end{proof}

\subsection{Characters bounds}
We now convert bounds on biased skew dimensions into character bounds.

\begin{proposition}\label{prop: character bound for high level representations and macroscopic support sizes}
Let $\delta>0$, $\eta>0$, $\rho >0$, and set $\theta = 2/3$. Let $\varepsilon_0>0$. There exists $N_3$ such that the following holds.  Let $n\geq N_3$, $\eta n \leq k \leq (1-\delta)n$, and set $f=n-k$. Let $\lambda\vdash n$ such that $\lambda_1'\leq \lambda_1 \leq (1-\rho)n$. Finally, let $\sigma \in \kS_n$ with $f$ fixed points. Then
\begin{equation}
    \bg\chi^{\lambda}(\sigma)\bd \leq d_\lambda^{-(2-\varepsilon_0)\frac{\ln(n/f)}{\ln n}}.
\end{equation}
\end{proposition}
\begin{proof}
    Set $t = \frac{1}{2-\varepsilon_0}\frac{\ln n}{\ln(n/f)}$. Since $f\geq \delta n$, we have $t \geq t_0 := \frac{1}{2}\frac{\ln n}{\ln(1/\delta)}$. We want to prove that $d_\lambda \bg\chi^{\lambda}(\sigma)\bd^t \leq 1$. 
    Here $k\geq \eta n$, so by Lemma \ref{lem: bound on MN rewritten with ALS renomalized and simplified lemme fusionné}, for $n$ large enough we have
    \begin{equation}
        \bg \chi^\lambda(\sigma) \bd  
        \leq e^{\Cpart \sqrt{n}} \max_{\mu \subset_{\vdash k} \lambda} B_\theta(\lambda, \mu).
    \end{equation}
As previously, let $(M_n)$ be a sequence of integers such that $1\lll M_n \lll \ln \ln n$, and denote the parameters of its depth-$M_n$ triple slicing by $(\lambda(i), n_i)_{1\leq i\leq 3}$.
Let $\varepsilon_1$ as in Lemma \ref{lem: bound on bsd leading to representations being negligible if all their hooks are small}, $\alpha$ as in Lemma \ref{lem: exponential bound with alpha intended for diagrams with small center} and set $\varepsilon' = \alpha$. We split the proof into three cases.
   
\fbox{$\lambda_1\leq \varepsilon_1 n$}

\medskip

\noindent By Lemma \ref{lem: bound on bsd leading to representations being negligible if all their hooks are small} and since $d_\lambda \leq \sqrt{n!}\leq n^{n/2}$, for $n$ large enough we have
\begin{equation}
\begin{split}
   d_\lambda \bg\chi^{\lambda}(\sigma)\bd^t \leq  d_\lambda \bg\chi^{\lambda}(\sigma)\bd^{t_0} & \leq n^{n/2} \pg  e^{\Cpart \sqrt{n}} \max_{\mu \subset_{\vdash k} \lambda} B_\theta(\lambda, \mu) \pd^{t_0} \\
   & \leq e^{O(\sqrt{n}\ln n)} n^{n/2} e^{-(5/2)n\ln n} = e^{(-2+o(1))n\ln n}\leq 1.
\end{split}
\end{equation}

\fbox{$\lambda_1> \varepsilon_1 n$ and $n_3 \leq \varepsilon' n$}

\medskip

\noindent Here we use Lemma \ref{lem: bound dimension M N} and Lemma \ref{lem: exponential bound with alpha intended for diagrams with small center}. Since $M_n = n^{o(1)}$, this gives
\begin{equation}
  d_\lambda \bg\chi^{\lambda}(\sigma)\bd^t \leq  d_\lambda \bg\chi^{\lambda}(\sigma)\bd^{t_0} \leq (4M_n)^n n^{n_3/2} e^{-\alpha n \ln n} \leq e^{(-\alpha/2 + o(1))n\ln n} \leq 1.
\end{equation}

\fbox{$\lambda_1> \varepsilon_1 n$ and $n_3 > \varepsilon' n$}

\medskip

\noindent By Lemma \ref{lem: bound dimension M N} and Lemma \ref{lem: f over n bound on biased dimensions for diagrams with large center}, and using again that $(4M_n)^n = e^{o(n\ln n)}$ and $e^{\Cpart\sqrt{n}t} = e^{o(n\ln n)} $, we obtain
\begin{equation}
    d_\lambda \bg\chi^{\lambda}(\sigma)\bd^t \leq e^{o(n\ln n)} n^{n_3/2} (f/n)^{n_3 \frac{1}{2-\varepsilon_0} \frac{\ln n}{\ln (n/f)}} = n^{\pg \frac{1}{2} - \frac{1}{2-\varepsilon_0} + o(1)\pd \ln n}
    \leq 1.
\end{equation}
Hence we have covered all cases, and the proof is complete.
\end{proof}

\section{Bounds for low-level representations}
\subsection{The strategy of Diaconis and Shahshahani for transpositions}
For transpositions we have the following elegant formula:
\begin{equation}\label{eq:magical formula characters transpositions}
    \chi^\lambda(\tau) = \frac{1}{\binom{n}{2}}\sum_i \left( \binom{\lambda_i}{2} - \binom{\lambda'_i}{2}\right),
\end{equation}
where $\lambda\vdash n$ and $\tau \sim [2, 1^{n-2}]$. One can also derive \eqref{eq:magical formula characters transpositions} from the Murnaghan--Nakayama rule and the Naruse hook length formula.
Finding asymptotic bounds for \eqref{eq:magical formula characters transpositions} for all diagrams is a key technical part of the proof of cutoff for random transpositions by Diaconis and Shahshahani~\cite{DiaconisShahshahani1981}. To do so, they partitioned the set of diagrams according to how many boxes they have over the first row, in other words they did a first row decomposition as in Figure \ref{fig:lambdadown1slicing}, and then used the monotonicity of the formula \eqref{eq:magical formula characters transpositions} to bound the characters. This led to a character bound which was sufficient if $\lambda$ has a long first row (that is, if $\lambda_1 \geq n/2$).

\begin{figure}[!ht]
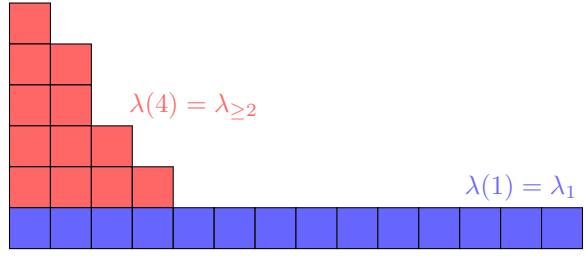

\begin{center}
\begin{ytableau}
        \none &*(red!60)\\
        \none &*(red!60)&*(red!60)\\
        \none &*(red!60)&*(red!60)&\none &\none&\none[\textcolor{red!60}{\lambda(4) = \lambda_{\geq 2}}]\\
        \none &*(red!60)&*(red!60)&*(red!60) \\
        \none &*(red!60)&*(red!60)&*(red!60)&*(red!60) &\none &\none&\none&\none&\none&\none&\none &\none&\none[\textcolor{blue!60}{\lambda(1) = \lambda_1}]\\
        \none &*(blue!60)&*(blue!60)&*(blue!60)&*(blue!60)&*(blue!60)&*(blue!60)&*(blue!60)&*(blue!60)&*(blue!60)&*(blue!60)&*(blue!60)&*(blue!60)&*(blue!60)&*(blue!60) &\none 
    \end{ytableau}
\end{center}
\caption{The first row decomposition for $\lambda = [14,4,3,2,2,1]$.}
\label{fig:lambdadown1slicing}
\end{figure}

\medskip

In this section we use hybrid notation. Let $\lambda$ be a Young diagram. We denote again by $\lambda(1)$ the first row of $\lambda$, $n_1 = |\lambda(1)|$ or $\lambda_1$ its size. Now we set $\lambda(4) = \lambda\backslash \lambda(1)$, and denote by $n_4 = |\lambda(4)|$ or $r(\lambda)$ its size.
Note that $\lambda(4)$ corresponds to $\lambda(2) \cup \lambda(3)$ from Section \ref{s: bounds for high-level representations}.

\medskip

For general conjugacy classes there is no usable formula such as \eqref{eq:magical formula characters transpositions}, nor monotonicity, and first row decompositions do not lead to character bounds that are sufficient for $0\leq n_4\leq n/2$. We can however prove bounds for diagrams $\lambda$ with a long enough first row. By symmetry of all our arguments (with respect to transposition of diagrams), the bounds extend to diagrams with a long first column.

\subsection{Splitting the characters into two parts}

Let $n$ and $k$ be integers such that $2\leq k\leq n-1$. Set $f = n-k$. Let $\lambda\vdash n$ such that $\lambda_1 \geq k$. Let $\sigma \in \kS_n$ with support size $k$. Recall the branching rule \eqref{eq:branching rule}:
\begin{equation}
    \chi^\lambda(\sigma) = \sum_{\mu \subset_{\vdash k} \lambda} \ch^\mu(\sigma^*) \frac{d_{\lambda\backslash \mu}}{d_\lambda}.
\end{equation}
For low-level representations, applying directly the triangle inequality in the branching rule does not lead to precise bounds. One has to first extract the two diagrams $\mu$ that matter (the most), namely $\mu = [k]$ and $\mu = [k-1,1]$, and crucially also take into account the sign compensation coming from $\mu = [k-1,1]$. We therefore set 
\begin{equation}\label{eq: def W Z}
    W(\lambda,\sigma) = \ch^{[k]}(\sigma^*) \frac{d_{\lambda \backslash [k]}}{d_\lambda} + \ch^{[k-1,1]}(\sigma^*) \frac{d_{\lambda \backslash [k-1,1]}}{d_\lambda} \quad \text{ and } \quad  Z(\lambda,\sigma) = \chi^\lambda(\sigma) - W(\lambda, \sigma).
\end{equation}

\begin{lemma}\label{lem: W form with excited sums}
Let $n$ and $k$ be integers such that $2\leq k\leq n-1$. Set $f = n-k$. Let $\lambda\vdash n$ such that $\lambda_1 \geq k$. Then
\begin{equation}
        W(\lambda, \sigma) = \frac{1}{n^{\downarrow k}}(S(\lambda, [k]) - S(\lambda, [k-1,1])).
    \end{equation}
\end{lemma}
\begin{proof}
    For any permutation $\tau\in \kS_k$, we have $\ch^{[k]}(\tau) = 1$ and $\ch^{[k-1,1]}(\tau) = f_1(\tau) - 1$, where we recall that $f_1(\tau)$ denotes the number of fixed points of $\tau$. Therefore $\ch^{[k]}(\sigma^*) = 1$ and $\ch^{[k-1,1]}(\sigma^*) = f_1(\sigma^*) - 1 = -1$ since $\sigma^*$ is fixed-point-free by definition. It follows that $W(\lambda, \sigma) = \frac{d_{\lambda \backslash [k]}}{d_\lambda} - \frac{d_{\lambda \backslash [k-1,1]}}{d_\lambda}$, which rewrites as desired by the Naruse hook length formula \eqref{e:NHLF rewritten}.
\end{proof}
The previous lemma shows that $W(\lambda, \sigma)$ depends only on the support size of $\sigma$. From now on we therefore write $W_\lambda(k)$ for $W(\lambda, \sigma)$.

\subsection{Bounding excited sums with geometric sums}

We first bound excited sums $S(\lambda, \mu)$ when $\mu$ is a row diagram.

\begin{lemma}\label{lem: exponential bound excited diagrams for row diagrams}
Let $n$ and $k$ be integers such that $2\leq k\leq n-1$. Set $f = n-k$. Let $\lambda\vdash n$ such that $\lambda_1 \geq k$. Denote $r = n-\lambda_1$. Assume that $r \leq f/4$ and let $0\leq \ell \leq k$. Then
    \begin{equation}
        S(\lambda, [\ell]) \leq e^{6r/f}\lambda_1^{\downarrow \ell}.  
    \end{equation}
\end{lemma}
\begin{proof}
First, partitioning the set of excited diagrams depending on how many boxes they have over the first row, we have 
\begin{equation}
    S(\lambda, [\ell]) = H(\lambda, [\ell]) + \sum_{1\leq i \leq \ell} H(\lambda, [\ell-i])S(\lambda, \nu(\ell, i)),
\end{equation}
where $\nu(\ell,i) = \Shift_{\ell-i+1, 1}([i])$, as represented in Figure \ref{fig:initial diagram for flat diagrams}. 

\begin{figure}[!ht]
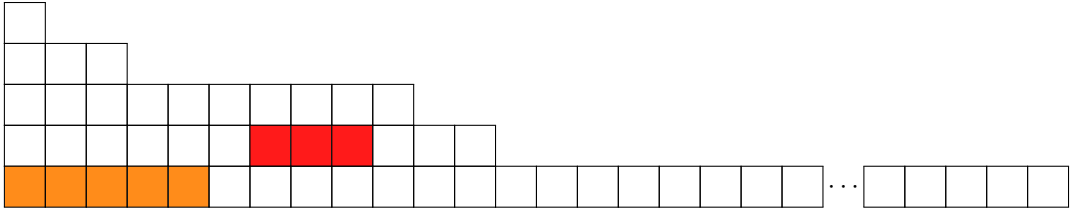

    \begin{center}
    \begin{ytableau}
        \none &\\
        \none &&&\\
        \none &&&&&&&&&&\\
        \none &&&&&&&*(red!90)&*(red!90)&*(red!90)&&& \\
        \none &*(orange!90)&*(orange!90)&*(orange!90)&*(orange!90)&*(orange!90)&&&&&&&&&&&&&&&&\none[\cdots]&&&&& \\
    \end{ytableau}
\end{center}

\caption{In orange: the diagram $[\ell-i]$ (with $\ell = 8$ and $i = 3$). In red: the diagram $\nu(\ell, i) = \Shift_{\ell-i+1, 1}([i])$. Both diagrams are represented within the diagram  $\lambda = [n-25, 12, 9,3,1]$.}
\label{fig:initial diagram for flat diagrams}
\end{figure}
\noindent Moreover, for each $0\leq i\leq \ell$, 
\begin{equation}
   S(\lambda, \nu(\ell, i)) = \frac{d_{\lambda\backslash \nu(\ell,i)}}{d_\lambda} r^{\downarrow i} \leq r^{\downarrow i}, 
\end{equation}
and by Lemma \ref{lem: slicing from ALS} (b), we have $H(\lambda, [\ell-i]) \leq e^{2r/n} \lambda_1^{\downarrow \ell-i} \leq e^{2r/f} \lambda_1^{\downarrow \ell-i}$.
Therefore for each  $0\leq i\leq \ell$ we have, using that $\lambda_1 - \ell = n-r -\ell \geq n-r-k=f-r \geq f/2$, 
\begin{equation}\label{eq: exponential bound excited diagrams for row diagrams inter}
    \frac{H(\lambda, [\ell-i])S(\lambda, \nu(\ell, i))}{\lambda_1^{\downarrow \ell}} e^{-2r/f} \leq  \frac{\lambda_1^{\downarrow \ell-i}r^{\downarrow i}}{\lambda_1^{\downarrow \ell}} = \frac{r^{\downarrow i}}{(\lambda_1-\ell + i)^{\downarrow i}} \leq \frac{r^{i}}{(\lambda_1-\ell + i)^{ i}} \leq \pg \frac{2r}{f}\pd^i.
\end{equation}
Also, we have $\sum_{i=0}^\ell (2r/f)^i \leq 1 + (2r/f)\sum_{i\geq 0}(1/2)^i = 1 + 4r/f \leq e^{4r/f}$, since $2r/f \leq 1/2$. We conclude that 
\begin{equation}
    \frac{S(\lambda, [\ell])}{\lambda_1^{\downarrow \ell}} = \frac{H(\lambda, [\ell])}{\lambda_1^{\downarrow \ell}} +  \sum_{i=1}^\ell \frac{H(\lambda, [\ell-i])S(\lambda, \nu(\ell, i))}{\lambda_1^{\downarrow \ell}} \leq e^{2r/f} \sum_{i=0}^\ell (2r/f)^i \leq e^{6r/f}. \qedhere
\end{equation}
\end{proof}
Let us generalize this to diagrams which may have boxes over the first row.

\begin{lemma}\label{lem: exponential bound excited diagrams}
    Let $n$ and $k$ be integers such that $2\leq k\leq n-1$. Set $f = n-k$. Let $\lambda\vdash n$ such that $\lambda_1 \geq k$ and $r \leq f/4$, where $r = n-\lambda_1$. Let $\mu \subset_{\vdash k} \lambda$, and denote $j = k-\mu_1$. Then
    \begin{equation}
        \frac{S(\lambda, \mu)}{\lambda_1^{\downarrow k}} \leq e^{6r/f} \pg \frac{2r}{f} \pd^j.
    \end{equation}
\end{lemma}
\begin{proof}
    Set $\ell := \mu_1 = k-j$. By Lemmas \ref{lem: fragmentation bound for excited sums} and \ref{lem: exponential bound excited diagrams for row diagrams}, we have
\begin{equation}
S(\lambda, \mu) \leq S(\lambda, [\ell])S(\lambda, \mu_{\geq 2}) \leq e^{6r/f} \lambda_1^{\downarrow \ell}  r^{\downarrow j}.
\end{equation}
Moreover, using again that $\lambda_1 - \ell + j \geq f/2 \geq r$, we get
\begin{equation}\label{eq: exponential bound excited diagrams factorielles descendantes}
    \frac{\lambda_1^{\downarrow \ell}  r^{\downarrow j}}{\lambda_1^{\downarrow k}} = \frac{r^{\downarrow j}}{(\lambda_1-\ell + j)^{\downarrow j}} \leq \pg \frac{2r}{f} \pd^j.
\end{equation}
This concludes the proof.
\end{proof}

\subsection{Bounds on the main term}

\begin{lemma}\label{lem: W est positif}
    Let $n$ and $k$ be integers such that $2\leq k\leq n-1$. Set $f = n-k$. Let $\lambda\vdash n$ such that $\lambda_1 \geq k$ and $r \leq f/6$, where $r = n-\lambda_1$. Then $W_\lambda(k)\geq 0$.
\end{lemma}
\begin{proof}
On the one hand, we have $S(\lambda, [k])/\lambda_1^{\downarrow k} \geq H(\lambda, [k])/\lambda_1^{\downarrow k} \geq 1$. On the other hand, by Lemma \ref{lem: exponential bound excited diagrams} we have
     \begin{equation}
         \frac{S(\lambda, [k-1,1])}{\lambda_1^{\downarrow k}} \leq e^{6r/f}\frac{2r}{f} \leq e \frac{2}{6}\leq 1.
     \end{equation}
It follows that $S(\lambda, [k]) - S(\lambda, [k-1,1])\geq 0$, and therefore that $W_\lambda(k)\geq 0$ by Lemma \ref{lem: W form with excited sums}.
\end{proof}

For the next lemma we need additional notation on Young diagrams. Let $\lambda$ be a Young diagram and $1\leq i\leq \lambda_1$. Then $u_i = u_i(\lambda) := H(\lambda, (i,1)) - (\lambda_1 - i + 1)$ is the number of boxes in the $i$-th column of $\lambda$ that are strictly above the box with coordinates $(i,1)$, as represented in Figure \ref{fig: definition u i}. For $1\leq j \leq \lambda_1$, we also denote $u_{\leq j} = \sum_{1\leq i\leq j} u_i$ and $u_{\geq j} = \sum_{j\leq i\leq \lambda_1} u_i$.

\begin{figure}
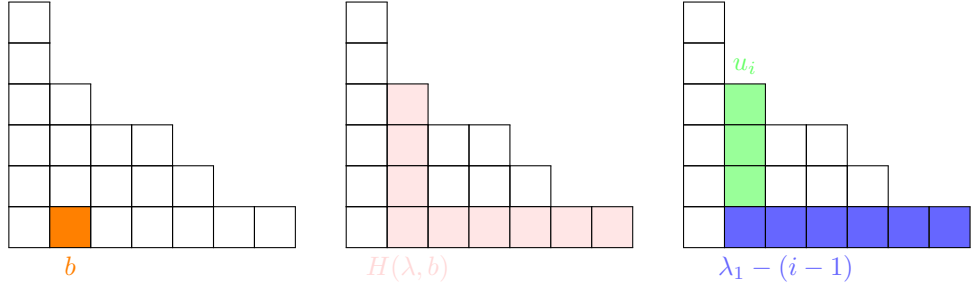

\begin{center}
\begin{ytableau}
    \none   & \\
    \none & \\
    \none   & &\\
    \none   & & & & \\
    \none   & & & & &\\
    \none   & & *(orange)& &&&& \\
    \none &\none  &\none[\textcolor{orange}{b}]  &\none  &\none  &\none  &\none  
\end{ytableau}
\begin{ytableau}
    \none   & \\
    \none & \\
    \none   & &*(pink!40)\\
    \none   & & *(pink!40) & & \\
    \none   & & *(pink!40) & & &\\
    \none   & & *(pink!40)& *(pink!40)& *(pink!40)& *(pink!40)& *(pink!40)& *(pink!40)  \\
    \none &\none  &\none[\textcolor{pink!60}{H(\lambda, b)}]  &\none  &\none  &\none  &\none  
\end{ytableau}
\begin{ytableau}
    \none   & \\
    \none & &\none[\textcolor{green!60}{u_i}]\\
    \none   & &*(green!40)\\
    \none   & & *(green!40) & & \\
    \none   & & *(green!40) & & &\\
    \none   & & *(blue!60)& *(blue!60)& *(blue!60)& *(blue!60)& *(blue!60)& *(blue!60)  \\
    \none &\none  &\none  &\none[\textcolor{blue!60}{\lambda_1-(i-1)}]  &\none  &\none  &\none  
\end{ytableau}
\end{center}

  \caption{Representation of $u_i$ for $i=2$ in the diagram $\lambda = [7,5, 4, 2,1,1].$}
\label{fig: definition u i}
\end{figure}

\begin{lemma}\label{lem: monotonicity lemma for low frequencies}
    Let $n$ and $k$ be integers such that $2\leq k\leq n-1$. Set $f = n-k$. Let $\lambda\vdash n$ such that $\lambda_1 \geq k$. Then
    \begin{equation}
        H(\lambda, [k-1])(\lambda_1 -(k-1) - u_{\leq k-1}) \leq \lambda_1^{\downarrow k}.
    \end{equation}
\end{lemma}
\begin{proof}
By definition we have $H(\lambda, [k-1]) = \prod_{1\leq i\leq k-1} (\lambda_1 - (i-1) + u_i)$. Therefore the quantity we want to bound can be rewritten as
 \begin{equation}
         \pg \prod_{1\leq i\leq k-1} \lambda_1 - (i-1) + u_i  \pd (\lambda_1 -(k-1) - u_{\leq k-1} ).
\end{equation}
Moreover moving a box of $\lambda$ from above the first row to the first row (weakly) increases all terms $\lambda_1 + u_i -(i-1)$, while $\lambda_1 -u_{\leq k-1}  - (k-1)$ does not change (since $\lambda_1$ increases by 1, but so decreases $u_{\leq k-1}$). Therefore the diagram $\lambda = [n]$ is maximal for the product above, and we conclude that
\begin{equation}
         \pg \prod_{1\leq i\leq k-1} \lambda_1 - (i-1) + u_i  \pd (\lambda_1 -(k-1) - u_{\leq k-1} ) \leq \lambda_1^{\downarrow k}. \qedhere
\end{equation}
\end{proof}

\begin{lemma}\label{lem: upper bound on W}
    Let $n$ and $k$ be integers such that $2\leq k\leq n-1$. Set $f = n-k$. Let $\lambda\vdash n$ such that $\lambda_1 \geq k$ and $r \leq f/6$, where $r = n-\lambda_1$. Then 
    \begin{equation}
        W_\lambda(k) \leq \pg \frac{f}{n} \pd^r e^{12r^2/f^2}.
    \end{equation}
\end{lemma}
\begin{proof}

Recall that $\nu(k,i) = \Shift_{k-i+1, 1}([i])$, as in Figure \ref{fig:initial diagram for flat diagrams}. Let us first upper bound $S(\lambda, [k])$. Observe that 
\begin{equation}
\begin{split}
     H(\lambda, [k]) + H(\lambda, [k-1]) S(\lambda, \nu(k,1)) & = H(\lambda, [k-1])S(\lambda, (k,1)) \\
      & = H(\lambda, [k-1])(\lambda_1 -(k-1) + u_{\geq k}).
\end{split}
\end{equation}
Then applying \eqref{eq: exponential bound excited diagrams for row diagrams inter} with $\ell = k$ for the terms $j\geq 2$ we get
\begin{equation}
\begin{split}
      S(\lambda, [k]) & \leq  H(\lambda, [k]) + H(\lambda, [k-1]) S(\lambda, \nu(k,1)) + \sum_{2\leq j \leq k} H(\lambda, [k-j]) S(\lambda, \nu(k,j)) \\
      & = H(\lambda, [k-1])(\lambda_1 -(k-1)+ u_{\geq k}) + \sum_{2\leq j \leq k} H(\lambda, [k-j]) S(\lambda, \nu(k,j)) \\
      & \leq H(\lambda, [k-1])(\lambda_1 -(k-1)+ u_{\geq k}) + e^{6r/f} \pg 2r/f \pd^2 \lambda_1^{\downarrow k}.
\end{split}
\end{equation}
On the other hand we have
\begin{equation}
    S(\lambda, [k-1,1]) \geq H(\lambda, [k-1])S(\lambda, (2,1)) = H(\lambda, [k-1])r.
\end{equation}
Since $r- u_{\geq k} = u_{\leq k-1}$, it follows from Lemma \ref{lem: monotonicity lemma for low frequencies} that
\begin{equation}
    S(\lambda, [k]) - S(\lambda, [k-1,1]) - e^{6r/f} \pg \frac{2r}{f} \pd^2 \lambda_1^{\downarrow k} \leq H(\lambda, [k-1])(\lambda_1 -(k-1)- u_{\leq k-1}) \leq  \lambda_1^{\downarrow k},
\end{equation}
which can be rewritten as
\begin{equation}
    W_k(\lambda) \leq \frac{\lambda_1^{\downarrow k}}{n^{\downarrow k}}\pg 1 + e^{6r/f} \pg \frac{2r}{f} \pd^2 \pd.
\end{equation}
Finally, we have 
\begin{equation}\label{eq: borne f sur n puissance r}
    \frac{\lambda_1^{\downarrow k}}{n^{\downarrow k}} = \frac{f^{\downarrow r}}{n^{\downarrow r}}\leq (f/n)^r,
\end{equation}
and since $r\leq f/6$, we also have 
\begin{equation}
    1 + e^{6r/f} \pg 2r/f \pd^2 \leq 1+ 3\pg 2r/f\pd^2 = 1+12 \pg r/f \pd^2 \leq e^{12r^2/f^2},
\end{equation}
which concludes the proof.
\end{proof}

\subsection{Bounds on the error term}

\begin{lemma}\label{lem: bound on biased skew dimensions for each j low level}
   Let $0\leq \theta \leq 1$. Let $n$ and $k$ be integers such that $2\leq k\leq n-1$. Set $f = n-k$. Let $\lambda\vdash n$ such that $\lambda_1 \geq k$ and $n-\lambda_1 \leq f/6$. Let $(\lambda(i), \mu(i), n_i, k_i)_{i\in \ag 1,4 \ad }$ be the parameters of the first row decompositions as in Figure~\ref{fig:lambdadown1slicing}. Then
    \begin{equation}
         B_\theta(\lambda, \mu) = d_\mu^\theta \frac{S(\lambda, \mu)}{n^{\downarrow k}} \leq e^{6n_4/f} \binom{k}{k_4}^{\theta} \frac{n_1^{\downarrow k_1} n_4^{\downarrow k_4}}{n^{\downarrow k}} B_\theta(\lambda(1), \mu(1))B_\theta(\lambda(4), \mu(4)).
    \end{equation}
In particular, recalling that $\theta_k = \frac{1}{2} \pg 1+ \frac{\Cvirt}{\ln k} \pd$ and assuming also that $k_4\geq 2$, we have 
\begin{equation}
    p(j) B_{\theta_{k}}(\lambda, \mu) \leq \pg \frac{f}{n}\pd^{r} \pg C_0 \frac{kr^2}{jf^2}\pd^{j/2}.
\end{equation}
where $r = n-\lambda_1 = n_4$, $j = |\mu_{\geq 2}| = k_4$, $C_0= 16e^{\Cvirt/2 + 2}$ is a universal constant, and we recall that $p(j)$ is the number of integer partitions of $j$.
\end{lemma}
\begin{proof}
    In this proof we use hybrid notation, with $n_4 = r$ and $k_4 = j$. The arguments are similar to those of the proof of Proposition \ref{prop: borne dimensions gauches biaisées par découpe profonde triple}.
    
    First, we have $S(\lambda, \mu) \leq S(\lambda, \mu(1)) S(\lambda, \mu(4)) = S(\lambda, \mu(1)) S(\lambda(4), \mu(4))$.
Moreover by Lemma~\ref{lem: exponential bound excited diagrams for row diagrams}, we have $S(\lambda, \mu(1)) \leq e^{6r/f} \lambda_1^{\downarrow \mu_1} \leq e^{6r/f} S(\lambda(1), \mu(1))$. Finally, $d_{\mu} \leq \binom{k}{j} d_{\mu(1)}d_{\mu(4)}$. The first point follows.
Now we have $\binom{k}{j} \leq k^j$ and $\binom{k}{j} \leq (ek/j)^j$, so
\begin{equation}
    \binom{k}{j}^{\theta_k} =  \binom{k}{j}^{\Cvirt/(2\ln k)}  \binom{k}{j}^{1/2} \leq \pg k^j\pd ^{\Cvirt/(2\ln k)} \pg (ek/j)^j\pd^{1/2} = \pg e^{\Cvirt/2 + 1} k/j\pd^{j/2}.
\end{equation}
Moreover, recalling \eqref{eq: exponential bound excited diagrams factorielles descendantes} and \eqref{eq: borne f sur n puissance r}, we have 
\begin{equation}
     \frac{n_1^{\downarrow k_1}n_4^{\downarrow k_4}}{n^{\downarrow k}} = \frac{n_1^{\downarrow k}}{n^{\downarrow k}} \frac{n_1^{\downarrow k_1} n_4^{\downarrow k_4}}{n_1^{\downarrow k}} = \frac{\lambda_1^{\downarrow k}}{n^{\downarrow k}} \frac{\lambda_1^{\downarrow k-j} r^{\downarrow j}}{\lambda_1^{\downarrow k}} \leq \pg \frac{f}{n}\pd^{r} \pg \frac{2r}{f} \pd^j
\end{equation}
and we have $e^{6r/f}\leq e\leq e^{j/2}$ by assumption. Finally, bounding the biased dimensions by 1 (recalling \eqref{eq: bound on biased skew dimensions by 1}), and since $p(j) \leq 2^j$, we conclude that
\begin{equation}
    p(j) B_{\theta_k}(\lambda, \mu) \leq 2^j e^{j/2} \pg e^{\Cvirt/2 + 1} \frac{k}{j}\pd^{j/2} \pg \frac{f}{n}\pd^{r} \pg \frac{2r}{f} \pd^j =  \pg \frac{f}{n}\pd^{r} \pg C_0 \frac{kr^2}{jf^2}\pd^{j/2}. \qedhere
\end{equation}
\end{proof}

\begin{lemma}
\label{lem: bornes technique en amont pour les sommes de B theta lambda mu à j fixé pour les diagrammes fins}
Let $\alpha > 0$.
Then,
\begin{equation}
    \sum_{j\ge2}
    (\alpha/j)^{j/2}
\le
    e^{8\alpha} - 1.
\end{equation}
\end{lemma}
\begin{proof}
Set for this proof
\begin{equation}
    S(\alpha) = \sum_{j\geq 2} (\alpha/j)^{j/2},
\quad\quad \text{and} \quad \quad
    u_\alpha(j) = (\alpha/j)^{j/2} \quad \text{ for any integer $j\geq 1$}.
\end{equation}
Assume first that $\alpha \le 1$.
Then $j\mapsto u_\alpha(j)$ is decreasing, so 
\begin{equation}\label{eq: diagrammes fins passage impairs à pair dans somme sous-géométrique}
    S(\alpha) = \sum_{j\geq 2} u_\alpha(j) = \sum_{m\geq 1} \pg u_\alpha(2m) + u_\alpha(2m+1) \pd\leq 2 \sum_{m\geq 1} u_\alpha(2m).
\end{equation}
Moreover, since $\alpha\leq 1$ we have
\begin{equation}
   \sum_{m\geq 1} u_\alpha(2m) = \sum_{m\geq 1} \pg \frac{\alpha/2}{m} \pd^m \leq \sum_{m\geq 1} \pg \alpha/2 \pd^m \leq \alpha \sum_{m\geq 1} \pg 1/2 \pd^m = \alpha,
\end{equation}
and we conclude that $S(\alpha) \leq 8\alpha \le e^{8\alpha} - 1$.

Assume now that $\alpha\geq 1$. Seen as a function over the positive real numbers, the function $j\in \bbR^*_+ \mapsto (\alpha/j)^{j/2}$ is maximized at $j = \alpha/e$, where it takes the value $e^{\alpha/(2e)}$. Therefore for any integer $j\geq 1$ we have $(\alpha/j)^{j/2} \leq e^{\alpha/(2e)}$.
Set $m_0 = 3 \lc \alpha/e \rc \in [\alpha, 3\alpha]$.
Then,
\begin{equation}
 \sum_{2\leq j\leq 2m_0 } u_\alpha(j) \leq 2m_0 e^{\alpha/(2e)} \leq 6\alpha e^{\alpha} \leq e^{7\alpha}. 
\end{equation}
Since $j\in [\alpha/e, \infty) \mapsto (\alpha/j)^{j/2}$ is decreasing, starting as in \eqref{eq: diagrammes fins passage impairs à pair dans somme sous-géométrique} and using that $m^m \geq m!$ we get the following bound for the tail of the sum:
\begin{equation}
   \sum_{j > 2m_0} u_\alpha(j) \leq 2\sum_{m\geq m_0} \pg \frac{\alpha/2}{m} \pd^m \leq \sum_{m\geq m_0} \pg \frac{\alpha}{m} \pd^m \leq \sum_{m\geq m_0} \frac{\alpha^m}{m!} \leq e^\alpha. 
\end{equation}
We conclude that
\begin{equation}
    S(\alpha) =
    \sum_{2\leq j \leq 2m_0} u_\alpha(j) + \sum_{j > 2m_0} u_\alpha(j)
\leq
    e^{7\alpha} + e^{\alpha} \leq e^{8\alpha} -1.
\qedhere
\end{equation}
\end{proof}

\begin{proposition}
\label{prop: bound on sums of biased dimensions for thin diagrams from level 2 for inner diagrams}
There exists a universal constant $C>0$ such that the following holds. Let $n$ and $k$ be integers such that $2\leq k\leq n-1$. Set $f = n-k$. Let $\lambda\vdash n$ such that $\lambda_1 \geq k$ and $r \leq f/6$, where $r = n-\lambda_1$. Let $\sigma \in \kS_n$ with support size $k$.
Then,
\begin{equation}
    \bg Z(\lambda, \sigma)\bd \leq  \pg \frac{f}{n}\pd^r
    \pg e^{C r^2 k / f^2} - 1 \pd,
\end{equation}
where $Z(\lambda, \sigma)$ was defined in \eqref{eq: def W Z}.
\end{proposition}

\begin{proof}
First, for each $j\geq 2$ there are at most $p(j)$ Young diagrams $\mu\subset_{\vdash k} \lambda$ with $j$ boxes over the first row, and therefore by Lemma \ref{lem: bound on biased skew dimensions for each j low level} we have
\begin{equation}
    \sum_{\mu \subset_{\vdash k} \lambda \du |\mu_{\geq 2}| = j } B_{\theta_k}(\lambda, \mu) \leq p(j) \max_{\mu \subset_{\vdash k} \lambda \du |\mu_{\geq 2}| = j } B_{\theta_k}(\lambda, \mu) \leq \pg C_0 \frac{kr^2}{jf^2}\pd^{j/2}\pg \frac{f}{n}\pd^r,
\end{equation}
where $C_0$ is the universal constant from Lemma \ref{lem: bound on biased skew dimensions for each j low level}.
We deduce from Lemma \ref{lem: bornes technique en amont pour les sommes de B theta lambda mu à j fixé pour les diagrammes fins}, applied with $\alpha = C r^2 k / f^2$, and setting $C = 8C_0$, that
\begin{equation}
    \bg Z(\lambda, \sigma)\bd \leq \pg \frac{f}{n}\pd^r
    \pg e^{C r^2 k / f^2} - 1 \pd,
\end{equation}
which concludes the proof.
\end{proof}

\subsection{Combined bound}

\begin{proposition}\label{prop: character bound low level with f over n}
    There exists a universal constant $C>0$ such that the following holds. Let $n$ and $k$ be integers such that $2\leq k\leq n-1$. Set $f = n-k$. Let $\lambda\vdash n$ such that $\lambda_1 \geq k$ and $r \leq f/6$, where $r = n-\lambda_1$. Let $\sigma \in \kS_n$ with support size $k$.
Then,
\begin{equation}
    \bg \chi^\lambda(\sigma) \bd \leq  \pg \frac{f}{n}\pd^r
    e^{C r^2 k / f^2}.
\end{equation}
\end{proposition}
\begin{proof}
    Recall that $\chi^\lambda(\sigma) = W(\lambda, \sigma) + Z(\lambda,\sigma) = W_\lambda(k) + Z(\lambda,\sigma)$ by definition. By the triangle inequality, and since $W_\lambda(k)\geq 0$ by Lemma \ref{lem: W est positif}, we get
    \begin{equation}
        \bg \chi^\lambda(\sigma)\bd \leq |W_\lambda(k)| + |Z(\lambda,\sigma)| = W_\lambda(k) + |Z(\lambda,\sigma)|.
    \end{equation}
Applying Lemma \ref{lem: upper bound on W} to $W_\lambda(k)$ and Proposition \ref{prop: bound on sums of biased dimensions for thin diagrams from level 2 for inner diagrams} to  $|Z(\lambda,\sigma)|$ then concludes the proof.
\end{proof}

\subsection{Character bound written in exponential form}
We now combine the previous results to prove character bounds for low-level representations. We set $\zeta(r) = \ln \max(1,r)$ for $r\geq 0$.
\begin{proposition}\label{prop: borne caractères basses fréquences avec terme qui aide et suffisamment de points fixes}

There exist universal constants $n_0\geq 3$, $0<c_0<1$ such that the following holds. Let $n$ and $k$ be integers such that $n\geq n_0$ and $2\leq k \leq n-1$, $\sigma \in \kS_n$ with support $k$, and $\lambda \vdash n$ such that $r:= n-\lambda_1 \leq c_0 \frac{f^2}{n}$, where $f =n-k$. Then 
    \begin{equation}
        d_\lambda \bg \chi^{\lambda}(\sigma) \bd^{\frac{\ln n}{\ln(n/f)}} \leq d_{\lambda}^{- \frac{\zeta(r)}{\ln n}} \quad \quad \text{ or equivalently } \quad \quad \bg \chi^{\lambda}(\sigma) \bd \leq d_\lambda^{-\frac{\ln(n/f)}{\ln n} \pg 1 + \frac{\zeta(r)}{\ln n} \pd}.
    \end{equation}
\end{proposition}
\begin{proof} 
    If $r=0$ then $d_\lambda = 1 = \chi^{\lambda}(\sigma)$. If $r=1$ then $d_{\lambda} = n-1\leq n$ and, $\ch^{\lambda}(\sigma) = f-1$ so $\bg\chi^{\lambda}(\sigma)\bd = \frac{f-1}{n-1} \leq \frac{f}{n}$, and we have
    \begin{equation}
        d_\lambda \bg \chi^{\lambda}(\sigma) \bd^{\frac{\ln n}{\ln(n/f)}} \leq n (f/n)^{\frac{\ln n}{\ln(n/f)}} = 1.
    \end{equation}
We may therefore assume that $r\geq 2$. Then we have $r! \geq r^{r/2}$. For the rest of the proof, we set
\begin{equation}
    t = \frac{\ln n}{\ln(n/f)}.
\end{equation}
By Proposition \ref{prop: character bound low level with f over n}, we have 
\begin{equation}
    \bg\chi^\lambda(\sigma)\bd \leq (f/n)^r
    e^{C r^2 k / f^2},
\end{equation}
where $C$ is a universal constant. Therefore, since $d_\lambda \leq n^r/\sqrt{r!} \leq n^r /r^{r/4}$, we get
\begin{equation}
    d_\lambda \bg \chi^{\lambda}(\sigma) \bd^{t} \leq \frac{n^r}{r^{r/4}} (f/n)^t \exp\pg C\frac{r^2k}{f^2} t \pd  = \exp\pg C\frac{r^2k}{f^2} t - \frac{1}{4}r\ln r\pd.
\end{equation}
Let $c_0 = \frac{1}{96C}$. In particular $c_0 \leq \frac{\ln 2}{32C}$.
We may assume that $f\geq \sqrt{n}$ since for $f\leq \sqrt{n}$ (and $c_0<1$) the condition $r\leq c_0f^2/n$ is trivial. Let us bound $C\frac{r^2k}{f^2} t$ for different values of $f$ and $r$.

  \medskip
  
\fbox{First case: $\sqrt{n} \leq f < n^{3/4}$}

\medskip

\noindent Here we have $t \leq 4$, so
    \begin{equation}
        C\frac{r^2k}{f^2} t \leq 4C\frac{r^2n}{f^2} \leq 4Cc_0 r \leq \frac{1}{8} r \ln 2 \leq \frac{1}{8} r \ln r.
    \end{equation}

    \fbox{Second case: $f \geq n^{3/4}$ and $r\geq n^{1/12}$}

\medskip

\noindent We have $t = - (\ln n)/\ln(1 - k/n) \le \tfrac kn \ln n$. Therefore 
    \begin{equation}
       C\frac{r^2 k}{f^2} t \leq  C\frac{r^2 n}{f^2} \ln n \leq Cc_0 r \ln n \leq 12 C c_0 r\ln r \leq \frac{1}{8} r\ln r.
    \end{equation}

\fbox{Third case: $f \geq n^{3/4}$ and $r< n^{1/12}$}

\medskip

\noindent
        For $n$ large enough we have $C \ln n \leq n^{1/12}$ and $e^{n^{-1/4}} \leq 2^{1/8}$, so by the same argument as before,
\begin{equation}
    C\frac{r^2 k}{f^2} t \leq  C\frac{r^2 n}{f^2} \ln n \leq n^{1/12} \frac{\pg n^{1/12}\pd^2 n}{\pg n^{3/4}\pd^2} = n^{-1/4} \leq \frac{1}{8} 2 \ln 2 \leq \frac{1}{8} r\ln r.
\end{equation}        

\fbox{Conclusion}

\medskip

\noindent
We have proved that in all cases we have $C\frac{r^2 k}{f^2} t  - \frac{1}{4} r\ln r\leq -\frac{1}{8} r\ln r$ (for $n$ large enough). We deduce that for $2\leq r \leq c_0 f^2/n$, using that  $d_\lambda \leq n^r = r^{r\frac{\ln n}{\ln r}}$, that is, $d_{\lambda}^{-\frac{\ln r}{\ln n}} \geq r^{-r}$, we have
    \begin{equation}
        d_\lambda \bg \chi^{\lambda}(\sigma) \bd^{t} \leq r^{-r/8} \leq d_{\lambda}^{- \frac{\ln r}{8 \ln n}}.
    \end{equation}
Putting this together with the cases $r=0$ and $r=1$ we conclude that for any $0\leq r\leq c_0 f^2/n$,
\begin{equation}
    d_\lambda \bg \chi^{\lambda}(\sigma) \bd^{\frac{\ln n}{\ln(n/f)}} \leq d_{\lambda}^{- \frac{\ln \max(1, r)}{8 \ln n}} = d_{\lambda}^{- \frac{\zeta(r)}{\ln n}}.
\qedhere
\end{equation}
\end{proof}

\begin{remark}
    The relevance of the previous proposition is to have some additional compensation, but the precise function (here we proved the statement with $r\mapsto \tfrac{1}{8}\ln\max(1,r)$) does not matter, as long as it tends to infinity. Adapting the arguments of the proof, one could show that the same statement holds when both $n$ and $r$ diverge, with $1/2 - o(1)$ in place of $1/8$. Informally, in terms of mixing times this means that the representations $\lambda$ such that $n-\lambda_1$ is close to $n^{\alpha}$ stop contributing to the $L^2$ distance to stationarity from time about $\pg 1 - \alpha/2 \pd \frac{\ln n}{\ln(n/f)}$, which is long before the mixing time. 
\end{remark}

\section{Bootstrapped and unified character bounds}\label{s: incrementation}

Different techniques can be tailored to bound characters depending on the diagram shape we want to study. For example, Féray and Śniady \cite{FeraySniady2011} developed a technique that gives better bounds on characters $\ch^\lambda(\sigma)$ if the outer hook $s(\lambda)$ and the support of $\sigma\in \kS_n$ are small. On the other hand, the bounds of Larsen and Shalev \cite{LarsenShalev2008} are precise for permutations with few fixed points and all diagrams $\lambda$, and are especially good for diagrams with a long first row and cycle structures close to involving a single cycle length (i.e.\ a cycle type close to $[n/m]^m$ for some fixed integer $m$).

So far, we proved bounds for permutations with $\omega(\sqrt{n})$ fixed points and diagrams $\lambda$ with a long first row, as well as permutations whose support and number of fixed points are of order $n$ and all diagrams $\lambda$. 
Our bounds for diagrams with a very long first row are already precise, but the arguments from the previous sections do not give uniform bounds for characters that have a small support. However, as we see in this section, probabilistic arguments allow the deduction of character bounds for permutations with a small support from those for larger support. 

The idea is the following: the distribution of a Markov chain with transition matrix $P$ at time $t$ (which is a multiple of $m$) is the same as that of the Markov chain with transition matrix $P^m$ at time $t/m$. For example if $\mu_t$ denotes the distribution of the random transposition walk after $t$ steps and $m= \lf \varepsilon n/2\rf$, then $\mu_m$ is conjugacy invariant and concentrated on permutations whose support is close to $\varepsilon n$, so $\mu_t = \mu_m^{*t/m}$ is the distribution of a product of $t/m$ random permutations with support close to $\varepsilon n$. This simple yet powerful idea dates back at least to \cite{KalbfleischLawless1985} and more complex versions were developed in comparison techniques for Markov chains \cite{DiaconisSaloff-Coste1993comparisongroups, DiaconisSaloff-Coste1993comparisonreversiblechains}.

In the context of conjugacy invariant walks on groups, it was an important tool at least in \cite{MullerSchlagePuchta2007precutoff}, in \cite{BerestyckiSengul2019} where it allowed studying the problem on a scale on which curvature behaves well, and in \cite{LarsenTiep2024FiniteClassicalGroups} where it recently allowed Larsen and Tiep to bootstrap character bounds for group elements with large \textit{support} (for a suitable generalization of the notion of support to groups other than the symmetric group) into uniform character bounds. Our approach is therefore similar to that of \cite{MullerSchlagePuchta2007precutoff} and \cite{LarsenTiep2024FiniteClassicalGroups}. However, in order to prove character bounds with error terms precise enough for cutoff, we apply bootstrapping only to high enough frequencies. 

\subsection{Characters and powers of eigenvalues}
Recall that if $\cC \in \Conj^*(\kS_n)$, that is, if $\cC$ is a non-trivial conjugacy class of $\kS_n$, the distribution of the simple random walk on $\Cay(\kS_n, \cC)$ started at $\Id$ and after $t$ steps is $\mu_t := \Unif_{\cC}^{*t}$, the $t$-fold convolution product of $\Unif_{\cC}$. 

More explicitly, let $t\geq 1$, $X_t \sim \mu_t$, $X_1 \sim \mu_1 = \Unif_{\cC}$, and denote by $P$ the transition matrix of the random walk on $\kS_n$ driven by $\mu_1$. Then the transition matrix of the random walk on $\kS_n$ driven by $\mu_t$ is nothing else than $P^t$. In terms of eigenvalues, this means that $\bbE\cg \chi^{\lambda}(X_1) \cd^t = \bbE\cg \chi^{\lambda}(X_t) \cd$. Therefore, we have
\begin{equation}\label{eq: identity characters eigenvalues}
    \chi^\lambda(\cC)^t = \pg \sum_{\sigma \in \kS_n} \mu_1(\sigma)\pd = \bbE \cg \chi^\lambda(X_1) \cd^t = \bbE \cg \chi^\lambda(X_t) \cd =  \sum_{\cC' \in \Conj(\mathfrak{S}_n)} \mu_t(\cC') \chi^\lambda(\cC').
\end{equation}
A similar identity appeared for instance in the definition of $g_1$
in \cite{MullerSchlagePuchta2007precutoff}, below Lemma 10 on Page 934.

\begin{lemma}\label{lem: borne générale caractères et valeurs propres}
    Let $n\geq 1$, $\cC \in \Conj^*(\mathfrak{S}_n)$, $\lambda \vdash n$, and $t\geq 0$.
    Set $\mu_t = \Unif_{\cC}^{*t}$. Let $B \subset \mathfrak{S}_n$, and denote its complement by $B^c$. Then we have 
\begin{equation}
    \bg \chi^\lambda(\cC)^t\bd \leq \mu_t(B^c) + \max_{\sigma\in B} \bg  \chi^\lambda(\sigma)\bd.
\end{equation}
\end{lemma}
\begin{proof}
Applying the law of total probability and the triangle inequality in \eqref{eq: identity characters eigenvalues}, we get 
    \begin{equation}
    \begin{split}
        \bg  \chi^\lambda(\cC)^t \bd & = \bg \sum_{\sigma \in B^c} \mu_t(\sigma) \chi^\lambda(\sigma) +  \sum_{\sigma \in B} \mu_t(\sigma) \chi^\lambda(\sigma)\bd \\
        &\leq \mu_t(B^c)\max_{\sigma\in B^c} \bg \chi^\lambda(\sigma) \bd + \mu_t(B)\max_{\sigma\in B} \bg \chi^\lambda(\sigma)\bd.
    \end{split} 
    \end{equation}
Using that $\max_{\sigma\in B^c} \bg \chi^\lambda(\sigma) \bd \leq 1$ and $\mu_t(B)\leq 1$ gives the desired result.
\end{proof}

\subsection{Multi-coupon collector estimates}

The goal of this subsection is to control the size of the support of the product of i.i.d.\ random permutations uniform from a given conjugacy class.
To this end, we introduce the following notation. Let $n,p,q,t\geq 1$ be integers. Let $S_{n,q}(1), S_{n,q}(2), \ldots$ be a sequence of i.i.d.\ uniformly random subsets of $[n] = \ag 1, 2, \ldots, n\ad$ of size $q$. Let 
\begin{equation}
    N_{n,p,q}(t) =  [p]\cap \bigcup_{i=1}^t  S_{n,q}(i)\, ,
\end{equation}
\begin{equation}
\begin{split}
    N^{(2)}_{n,p,q}(t) & = \ag (i,j)\in [t] \times [p] \du  j \in S_{n,q}(i), \text{ and }  j \in S_{n,q}(i') \text{ for some } i'<i \ad \\
    & = \ag (i,j) \in [t] \times [p] \du j \in ( \cup_{i'=1}^{i-1} S_{n,q}(i') ) \cap S_{n,q}(i) \ad,
\end{split}
\end{equation}
and
\begin{equation}
    M^{(2)}_{n,p,q}(t) = \bg N^{(2)}_{n,p,q}(t)\bd = \sum_{i=1}^t \bg [p]\cap\left( \bigcup_{i'=1}^{i-1} S_{n,q}(i') \right)\cap S_{n,q}(i) \bd.
\end{equation}
In words, $M^{(2)}_{n,p,q}(t)$ is the number of times elements of $[p]$ that were already picked previously in the process $S_{n,q}(1), S_{n,q}(2), \ldots, S_{n,q}(t)$ are picked again afterwards, that is, the number of repetitions (possibly counting the same element of $[p]$ multiple times). In particular, one has
\begin{equation}
    M^{(2)}_{n,p,q}(t)
    = qt-\bg [p]\cap \bigcup_{i=1}^t S_{n,q}(i) \bd.
\end{equation}

If $X$ and $Y$ are two $\bbR$-valued random variables, we say that $X$ is (weakly) \textit{stochastically dominated}
by $Y$, and write $X \preceq Y$, if for every $x\in \bbR$, $\bbP(X\geq x) \leq \bbP(Y\geq x)$. We may also write stochastic domination with the laws of the random variables instead of the variables themselves. 

\begin{lemma}\label{lem: elementary bin stochastic domination}
    Let $n,p, q$ be integers such that $1\leq p\leq n$ and $1\leq q \leq n/2$. Then $[p]\cap S_{n,q}(1)\preceq \Bin(q,2p/n)$.
\end{lemma}
\begin{proof}
    By definition, $S_{n,q}(1)$ is a random subset of $[n]$ with $q$ elements. This can be viewed as labelling $n$ balls $1$ through $n$, and picking $q$ of them, one by one and without replacement. At the $i$-th draw (where $1\leq i\leq q$),  the probability of picking a ball whose label is in $[p]$ is at most $p/(n-(i-1)) \leq 2p/n$, and this bound holds independently of whether the previously picked balls were in $[p]$. 
   Therefore, $[p] \cap S_{n,q}(1)$ is stochastically dominated by a sum of $q$ independent Bernoulli variables of parameter $2p/n$, that is, by a $\Bin(q, 2p/n)$ random variable.
\end{proof}

This allows us to stochastically dominate $M_{n,n,q}^{(2)}(t)$ by a binomial law.

\begin{lemma}\label{lem: stochastic domination M 2 by bin}
    Let $n,q,t\geq 1$ be integers such that $qt\leq n/2$. Let $X\sim \Bin(qt,2qt/n)$. Then $M^{(2)}_{n,n,q}(t)\preceq X$.
\end{lemma}

\begin{proof}
Consider again $n$ balls, labelled $1$ through $n$. Let $m\geq 1$ be an integer. Set $X_i = N_{n,n,q}(i)\backslash N^{(2)}_{n,n,q}(i)$ and $x_i = |X_i|$ for $1\leq i\leq t$. By definition, recalling that we use the symbol $\sqcup$ to emphasize that a union is disjoint, we have
\begin{equation}
  M^{(2)}_{n,n,q}(t) = \bg  N^{(2)}_{n,n,q}(t) \bd= \bg \bigsqcup_{i=2}^t N^{(2)}_{n,n,q}(i) \backslash N^{(2)}_{n,n,q}(i-1) \bd = \sum_{i=2}^t  \bg X_{i-1} \cap S_{n,q}(i) \bd
\end{equation}
Since for each $2\leq i\leq t$, $S_{n,q}(i)$ is independent of $(S_{n,q}(j))_{1\leq j \leq i-1}$, and $|X_i| \leq \bg N_{n,n,q}(i)\bd \leq qt$, the variables  $\bg \cg qt \cd \cap S_{n,q}(i) \bd$ for $1\leq i\leq t$ are independent, and we have
\begin{equation}
    M^{(2)}_{n,n,q}(t) \preceq \sum_{i=1}^t  \bg \cg qt \cd \cap S_{n,q}(i) \bd.
\end{equation}
Moreover, $\bg \cg qt \cd \cap S_{n,q}(i) \bd \preceq \Bin(q, 2qt/n)$ for each $1\leq i \leq t$ by Lemma \ref{lem: elementary bin stochastic domination}.
We conclude that $M^{(2)}_{n,n,q}(t)$ is stochastically dominated by a sum of $t$ independent $\Bin(q, 2qt/n)$ random variables, i.e.\ by a $\Bin(qt, 2qt/n)$ random variable.
\end{proof}

\begin{lemma}\label{lem: dom sto support bin}
    Let $n,k, t$ be integers such that $2\leq k \leq n$,  $t \geq 2$, and $kt\leq n/2$. Let $\cC$ be a conjugacy class of $\kS_n$ with support $k$. Let $Z \sim \Unif_{\cC}^{*t}$ and recall that $\supp(Z)$ denotes the support size of $Z$. Then $kt-\supp(Z) \preceq \Bin(qt, 2qt/n)$.
\end{lemma}

\begin{proof}
First, by definition, $Z\sim \sigma_1\cdots\sigma_t$, where $\sigma_1, \ldots, \sigma_t$ are i.i.d.\ uniform permutations in $\cC$.  
Moreover, the elements in the support of the product permutation $Z$ consists at least of the integers that are in the support of exactly one $\sigma_i$, whose distribution is that of $kt - M^{(2)}_{n,n,k}(t)$.
Hence $\supp(Z) \succeq kt - M^{(2)}_{n,n,k}(t)$, which is equivalent to $kt - \supp(Z) \preceq M^{(2)}_{n,n,k}(t)$.  Moreover $M^{(2)}_{n,n,k}(t) \preceq \Bin(qt, 2qt/n)$ by Lemma \ref{lem: stochastic domination M 2 by bin}. We conclude that $kt - \supp(Z) \preceq \Bin(qt, 2qt/n)$.
\end{proof}

\subsection{Concentration of support}
We start by establishing a bound on the tails of binomial random variables.
\begin{lemma}\label{lem: borne binomiale eps alpha A}
     Let $0<\varepsilon, \alpha <1/2$, and $A\geq 2$ such that $\alpha \leq \frac{\varepsilon}{e}e^{-A}$. Let $n\geq 1/\varepsilon$ be an integer and $X \sim \Bin(n, \alpha )$.  Then
    \begin{equation}
        \bbP(X \geq \varepsilon n) \leq  e^{-(\varepsilon A/2) n}.
    \end{equation}
\end{lemma}
\begin{proof}
    For each $j\geq \lc \varepsilon n \rc$ we have $\binom{n}{j}\alpha ^j \leq (en/j)^j \alpha ^j \leq (e\alpha /\varepsilon)^j \leq e^{-jA}$.
It follows that
\begin{equation}
        \bbP(X \geq \varepsilon n)
        = \sum_{j\geq \lc \varepsilon n \rc} \binom{n}{j}\alpha ^j(1-\alpha )^{n-j} \leq \sum_{j\geq \lc \varepsilon n \rc} \binom{n}{j}\alpha ^j \leq \sum_{j\geq \lc \varepsilon n \rc}e^{-jA} = \frac{e^{-\lc \varepsilon n \rc A}}{1-e^{-A}}, 
    \end{equation}
and we conclude, since $e^{-A} \leq 1/2$ and $e^{-\varepsilon (A/2) n} \leq e^{-\varepsilon n} \leq e^{-1} \leq 1/2 $, that
\begin{equation}
   \bbP(X \geq \varepsilon n) \leq 2e^{-\lc \varepsilon n \rc A} \leq  2e^{-\varepsilon A n} \leq e^{-(\varepsilon A/2) n}. \qedhere
\end{equation}
\end{proof}
We now combine this with the multi-coupon collector arguments to show that the support of a product of $t$ invariant permutations in $\kS_n$ with support $k$ is close to $kt$, as long as $kt$ is small. The scales we consider are the same as for the curvature argument of \cite{BerestyckiSengul2019}. We emphasize that here $t$ is of order order $n/k$, which is not the same as the mixing time, which is close to $\frac{n}{k}\ln n$.

\begin{lemma}\label{lem: concentration du support}
   Let $0<\varepsilon, \alpha <1/2$, and $A\geq 2$ such that $\alpha \leq \frac{\varepsilon}{e}e^{-A}$. Let $n,k, t$ be integers such that $2\leq k \leq n$, $t \geq 2$, and $kt\leq \alpha n/2$. Let $\cC$ be a conjugacy class of $\kS_n$ with support $k$. Let $Z \sim \Unif_{\cC}^{*t}$. Set $y = \varepsilon k t$.
   Then
   \begin{equation}
       \bbP(\supp(Z) \leq kt-y) = 1- \bbP(kt - y < \supp(Z) \leq kt) \leq e^{-(\varepsilon A/2) kt}.
   \end{equation}
   \end{lemma}

\begin{proof}
We always have $\supp(Z) \le k t$ so the equality holds.
By Lemma \ref{lem: dom sto support bin} and since $2kt/n\leq \alpha$ by assumption, we have
\begin{equation}
    \supp(Z) \preceq \Bin(kt,2kt/n) \preceq \Bin(kt, \alpha).
\end{equation}
Combining this with Lemma \ref{lem: borne binomiale eps alpha A}, we finally get
\begin{equation}
    \bbP(kt-\supp(Z) \geq y) \leq \bbP(X \geq y) \leq e^{-(\varepsilon A/2) kt}.
\qedhere
\end{equation}
\end{proof}

\subsection{Bootstrapped bounds for high-level representations}

The following proposition extends the support condition $n\lesssim k \leq (1-\delta)n$ from Proposition \ref{prop: character bound for high level representations and macroscopic support sizes}, to all support sizes $2\leq k \leq (1-\delta)n$.

\begin{proposition}\label{prop: borne caractères incrémentée diagrammes intermédiaires}
Let $\delta>0$, $\rho>0$, and $\varepsilon_0>0$. There exists $n_0\geq 1$ such that the following holds. Let $n\geq n_0$, $k$ be an integer such that $2 \leq k \leq (1-\delta)n$, $\cC$ be a conjugacy class of $\kS_n$ with support $k$, and $\lambda \vdash n$ satisfying $\lambda_1'\leq \lambda_1 \leq (1-\rho)n$.
Then, setting $f:=n-k$,
   \begin{equation}
    d_\lambda \bg \chi^\lambda(\cC) \bd^{\pg\frac{1}{2-\varepsilon_0}\pd\frac{\ln n}{\ln(n/f)}} \leq 1, \quad \text{ or equivalently } \quad \bg \chi^\lambda(\cC) \bd \leq d_{\lambda}^{-(2-\varepsilon_0)\frac{\ln(n/f)}{\ln n}}.
\end{equation}
\end{proposition}

\begin{proof}
    Without loss of generality we may assume that $\varepsilon_0\leq 1/1000$, and that $1/\varepsilon_0$ is an integer. Set $\varepsilon = \varepsilon_0/12$, and let $q = 1/\varepsilon$, which is an integer. If $n/q^2 \leq k \leq (1-\delta)n$, then by Proposition \ref{prop: character bound for high level representations and macroscopic support sizes}, for $n$ large enough we have
\begin{equation}\label{eq: borne caractères incrémentée diagrammes intermédiaires equation intermédiaire}
    d_\lambda \bg \chi^\lambda(\cC) \bd^{\pg\frac{1}{2-\varepsilon}\pd\frac{\ln n}{\ln(n/f)}} \leq 1,
\end{equation}
and in particular this holds for $n/q^2 \leq k \leq (q+1)n/q^2$.

Assume from now on that $k \leq n/q^2$, and set $t = \lc \frac{n}{kq} \rc$. Since $\frac{n}{kq^2} \geq 1$, we have $\frac{1}{q}\frac{n}{k} \leq t \leq \frac{q+1}{q^2} \frac{n}{k}$. Consider the event $B = \ag \sigma \in \kS_n \du kt(1-1/q)  \leq \supp(\sigma) \leq kt \ad$, and let $\sigma \in B$. Then we have 
    \begin{equation}
        \frac{f_1(\sigma)}{n}  = 1- \frac{\supp(\sigma)}{n} \leq 1- \pg 1 - \frac{1}{q} \pd \frac{kt}{n} \leq 1- \pg 1 - \frac{1}{q} \pd\frac{1}{q} \leq e^{-1/q + 1/q^2}.
    \end{equation}
Moreover, 
\begin{equation}
    1-\frac{kt}{n} \geq 1 - \frac{q+1}{q^2} = 1 - \frac{1}{q} - \frac{1}{q^2} \geq e^{-1/q - 2/q^2},
\end{equation}
where we used that $q$ is large for the last inequality. It follows that
\begin{equation}
    \frac{\ln(n/f_1(\sigma))}{\ln(n/(n-kt))} \geq \frac{\ln ( e^{1/q -1/q^2} )}{\ln ( e^{1/q +2/q^2} )} = \frac{1-1/q}{1+2/q} \geq 1-\frac{4}{q} = 1- 4\varepsilon.
\end{equation}
Recall that $kt \in [n/q, (q+1) n/q^2]$. Then applying \eqref{eq: borne caractères incrémentée diagrammes intermédiaires equation intermédiaire} we obtain that for any $\sigma \in B$ we have, setting $z =1-kt/n$,
\begin{equation}
\label{eq:première équation}
    \bg \chi^\lambda(\sigma) \bd \leq d_{\lambda}^{-(2-\varepsilon) \frac{\ln(n/f_1(\sigma))}{\ln n}} \leq d_{\lambda}^{- \pg \frac{2-10\varepsilon}{1-4\varepsilon} \pd \frac{\ln(n/f_1(\sigma))}{\ln n}} \leq d_{\lambda}^{- \pg 2-10\varepsilon\pd \frac{\ln(1/z)}{\ln n}}
\end{equation}
Applying Lemma \ref{lem: concentration du support} with $A=20q$ and $\alpha=\varepsilon e^{-1-A}$, and recalling that $kt/q = \varepsilon kt$, we then get
\begin{equation}
\label{eq:deuxième équation}
    \Unif_{\cC}^{*t}(B^c) \leq e^{-(\varepsilon A/2) kt} = e^{-10kt} \leq e^{-5n\ln(1/z)}. 
\end{equation}
Rewriting $e^{5n}$ as $(n^{n/2})^{10/\ln n}$, we obtain $\Unif_{\cC}^{*t}(B^c) \leq (n^{n/2})^{-10 \frac{\ln(1/z)}{\ln n}} \leq d_\lambda^{-10 \frac{\ln(1/z)}{\ln n}}$.
Moreover, we have $\tfrac{\ln(1/z)}{\ln n}\asymp \tfrac{1}{\ln n}$, and $d_\lambda = e^{\Omega_+(n)} = n^{\omega_+(1)}$ by assumption, so $d_\lambda^{\frac{\ln(1/z)}{\ln n}} = o(1)$. It follows from Lemma \ref{lem: borne générale caractères et valeurs propres} that for $n$ large enough
\begin{equation}
    |\chi^\lambda(\cC)|^t \leq d_{\lambda}^{- 10 \frac{\ln(1/z)}{\ln n}} + d_{\lambda}^{- (2-10\varepsilon) \frac{\ln(1/z)}{\ln n}} \leq d_{\lambda}^{- (2-12\varepsilon) \frac{\ln(1/z)}{\ln n}} = d_{\lambda}^{- (2-\varepsilon_0) \frac{\ln(1/z)}{\ln n}}, 
\end{equation}
and therefore, since $-\ln(1/z) = \ln(1-kt/n) \leq t \ln(1-k/n) = t\ln(f/n) = -t\ln(n/f)$ (using that $t\geq 1$), we conclude that
\begin{equation}
    |\chi^\lambda(\cC)| \leq d_{\lambda}^{- \frac{1}{t} (2-\varepsilon_0) \frac{\ln(1/z)}{\ln n}} \leq d_{\lambda}^{-(2-\varepsilon_0) \frac{\ln(n/f)}{\ln n}}. \qedhere
\end{equation}
\end{proof}
\subsection{Unified character bounds}\label{s: unified character bounds}
We now have all the tools to prove our main character bounds. In this section we prove a stronger version of Theorem \ref{thm: character bound with no error nor helping term}, and then deduce Theorems~\ref{thm: uniform bound 1/2 intro} and~\ref{thm: character bound with no error nor helping term} from it.
Recall that the (symmetrized) level of a diagram $\lambda$ is defined as $\rsym(\lambda) = n - \max(\lambda_1, \lambda_1')$, and that if $\lambda_1 \geq \lambda_1'$, then $\rsym(\lambda) = n-\lambda_1$.
We recall that $\zeta(r) = \ln \max(1, r^{1/8})$ for $r\geq 0$.
\begin{theorem}\label{thm: character bound with helping term}
   Let $\delta>0$. There exists $n_0\geq 2$ such that for every $n\geq n_0$, $\lambda \in \widehat{\kS_n}$, and $\sigma \in \kS_n$ such that $f(\sigma)\geq \delta n$, we have 
    \begin{equation}
        \bg \chi^{\lambda}(\sigma) \bd \leq d_\lambda^{-\frac{\ln(n/f(\sigma))}{\ln n} \pg 1 + \frac{\zeta(\rsym(\lambda))}{\ln n} \pd}.
    \end{equation}
\end{theorem}

\begin{proof} 
Let $n$ be a large integer, $\lambda\vdash n$, and $\sigma \in \kS_n$. By symmetry (Lemma \ref{lem: identity characters sign transpose conjugation}) we may assume that $\lambda_1'\leq \lambda_1$, and set $r:= \rsym(\lambda) = n-\lambda_1$. Let $c_0$ be the constant from Proposition \ref{prop: borne caractères basses fréquences avec terme qui aide et suffisamment de points fixes}. If $r\leq c_0\delta^2n$, then the result holds by Proposition \ref{prop: borne caractères basses fréquences avec terme qui aide et suffisamment de points fixes}.
Assume on the other hand that $r> c_0\delta^2 n$. Set $\rho = c_0\delta^2$ and $\varepsilon_0 = 7/8$. Then since $1 + \frac{\zeta(r)}{\ln n} \leq 1+ 1/8 = 2 -\varepsilon_0$, the result holds by Proposition \ref{prop: borne caractères incrémentée diagrammes intermédiaires}. This concludes the proof.
\end{proof}

\begin{proof}[Proof of Theorem \ref{thm: character bound with no error nor helping term}]
 Theorem \ref{thm: character bound with no error nor helping term} follows from Theorem \ref{thm: character bound with helping term}, since $\zeta(r) \geq 0$ for any $r\geq 0$.   
\end{proof}

\begin{proof}[Proof of Theorem \ref{thm: uniform bound 1/2 intro}]
Recall that if $\sigma$ is a permutation $f_1(\sigma)$ denotes its number of fixed points and $f(\sigma) = \max(f_1(\sigma), 1)$.
By \cite[Example 1.4 (b) and (c)]{TeyssierThévenin2025virtualdegreeswitten}, for $\sigma \in \kS_n$ we have
\begin{equation}
    E(\sigma) - 1 \leq -\frac{1}{2}\frac{\ln(n/f(\sigma))}{\ln n}.
\end{equation}
It follows from Lemma \ref{lem:ALS improved character bound} that for all $n\geq 2$, $\sigma \in \kS_n$, and $\lambda\in \widehat{\kS_n}$, we have
\begin{equation}\label{eq: bound ALS inter for uniform bound with constant 1/2}
    \bg \chi^\lambda(\sigma)\bd
    \leq d_\lambda^{E(\sigma) - 1 + \frac{\Cvirt}{\ln n}} \leq d_\lambda^{-\frac{1}{2}\frac{\ln(n/f(\sigma))}{\ln n} + \frac{\Cvirt}{\ln n}} = d_\lambda^{-\frac{1}{2}\frac{\ln(n/f(\sigma))}{\ln n} \pg 1-\frac{2\Cvirt}{\ln(n/f(\sigma))} \pd}.
\end{equation}

Let $\varepsilon>0$. Let $\delta>0$ such that $2\Cvirt/\ln(1/\delta) \leq \varepsilon$.
Assume that $n$ is large, and let $\lambda\vdash n$.
Then applying \eqref{eq: bound ALS inter for uniform bound with constant 1/2} to permutations with at most $\delta n$ fixed points and Theorem \ref{thm: character bound with no error nor helping term} to permutations with at least $\delta n$, we obtain that for any $\sigma\in\kS_n$ we have
\begin{equation}
    \bg \chi^\lambda(\sigma)\bd
    \leq d_\lambda^{ -\frac{1-\varepsilon}{2}\frac{\ln(n/f(\sigma))}{\ln n}}.
\end{equation}
Since this holds for any $\varepsilon>0$ and all $n$ large enough (depending on $\varepsilon$), this concludes the proof.
\end{proof}

\section{Cutoff for conjugacy invariant random walks}

\subsection{A classical technique}\label{s: a classical technique}

Converting character bounds into mixing time estimates dates back to Diaconis and Shahshahani \cite{DiaconisShahshahani1981}. The idea consists of first bounding the total variation distance, which we recall is half of the $L^1$ distance, by the $L^2$ distance, and then applying the Plancherel identity to rewrite the $L^2$ distance in terms of characters and dimensions of irreducible representations. This gives a correspondence between eigenvalues and characters, and between multiplicities and squares of dimensions of irreducible representations.
This way of bounding the total variation distance is often referred to as the Diaconis--Shahshahani upper bound lemma, see also \cite[Lemma 3.B.1 on page 24]{LivreDiaconis1988}. Set $\widehat{\mathfrak{S}_n}^{**} = \widehat{\mathfrak{S}_n}\backslash\ag [n], [1^n] \ad$. For conjugacy classes $\cC$ of $\kS_n$, the Diaconis--Shahshahani upper bound lemma can be adapted to the coset distance to stationarity, see for instance \cite[Section 2]{Hough2016} and Remark \ref{rem: constante racine de deux} regarding the additional constant $1/2$:
\begin{equation}\label{eq: coset Diaconis--Shahshahani upper bound lemma}
   4\dtv(\cC, t)^2 \leq \dLtwo(\cC, t)^2 = \frac{1}{2}\sum_{\lambda\in \widehat{\mathfrak{S}_n}^{**}} d_\lambda^2 \bg \chi^\lambda(\cC)\bd^{2t}.
\end{equation}

\begin{remark}\label{rem: constante racine de deux}
   In this paper we defined the coset distances in \eqref{eq: def dtv coset} and \eqref{eq: def dLtwo coset} with respect to the set $\kE_n(\cC, t)$ rather than $\kS_n$. 
It enables not losing an extra factor $\sqrt{2}$ when applying the Cauchy--Schwarz inequality to bound the $L^1$ distance (which is twice the total variation distance) by the $L^2$ distance, and having the same coset $L^2$ profile as the lazy $L^2$ profile, at least for conjugacy classes with small support.
\end{remark}

A second important bound is the following. Liebeck and Shalev \cite{LiebeckShalev2004} showed that, for every $s>0$, as $n\to \infty$:
\begin{equation}\label{eq:bound Liebeck Shalev Witten zeta function}
     \sum_{\lambda \in \widehat{\mathfrak{S}_n}^{**}} d_\lambda^{-s} = O\pg n^{-s}\pd.
\end{equation}
Sums in the form of the left hand side of \eqref{eq:bound Liebeck Shalev Witten zeta function} are referred to as \textit{Witten zeta functions}, in reference to the work of Witten \cite{Witten1991gaugetheory}. Liebeck and Shalev \cite{LiebeckShalev2005} also proved bounds on the Witten zeta functions of other groups.

\medskip

Exponential character bounds are particularly convenient for applications. Informally, plugging a bound of the form $\bg \chi^\lambda(\cC) \bd \leq d_\lambda^{1/t_0}$ into \eqref{eq: coset Diaconis--Shahshahani upper bound lemma} implies, using \eqref{eq:bound Liebeck Shalev Witten zeta function}, that $\dtv(\cC, t) = o(1)$ for $t$ a little larger than $t_0$. This is exactly how we will deduce Theorem \ref{thm: cutoff intro} from Theorem \ref{thm: character bound with no error nor helping term}.

\subsection{Cutoff for conjugacy classes}
We now prove Theorem \ref{thm: cutoff intro}.
\begin{proof}[Proof of Theorem \ref{thm: cutoff intro}]
Let $0<\varepsilon<1$. Let $n$ be a large integer and $\cC\in \Conj^*(\kS_n)$ such that $f(\cC) \geq \delta n$.

Let us first prove the upper bounds. Set $t = \lc (1+\varepsilon) t_{\cC}\rc$. Since the $L^2$ distance is larger than the total variation distance, it is enough to upper bound the $L^2$ distance, and hence the right hand side of \eqref{eq: coset Diaconis--Shahshahani upper bound lemma}.
By Theorem \ref{thm: character bound with no error nor helping term}  and since $t\geq (1+\varepsilon) \frac{\ln n}{\ln(n/f(\cC))}$, we have
\begin{equation}
    2\dLtwo(\cC, t) = \sum_{\lambda\in \widehat{\mathfrak{S}_n}^{**}} d_\lambda^2 \bg \chi^\lambda(\cC)\bd^{2t} \leq \sum_{\lambda\in \widehat{\mathfrak{S}_n}^{**}} d_\lambda^2  \pg d_\lambda^{\frac{\ln(n/f(\cC))}{\ln n}}\pd^{2(1+\varepsilon)\frac{\ln n}{\ln(n/f(\cC))}} = \sum_{\lambda\in \widehat{\mathfrak{S}_n}^{**}} d_{\lambda}^{-2\varepsilon}.
\end{equation}
By the Liebeck--Shalev bounds on the Witten zeta function, recalled as \eqref{eq:bound Liebeck Shalev Witten zeta function}, we conclude that $\dLtwo(\cC, t) = O\pg n^{-2\varepsilon}\pd = o(1)$.

Let us now prove the lower bounds. Set $t = \lf (1-\varepsilon) t_{\cC}\rf$ and let $X_t \sim \Unif_{\cC}^{*t}$. The number of fixed points of a uniform permutation in $\kE(\cC, t)$ is asymptotically $\Poiss(1)$, and is therefore $< n^{\varepsilon/3}$ with probability $1-o(1)$. On the other hand, by classical coupon collector arguments, $X_t$ has $\geq n^{\varepsilon/3}$ fixed points with probability $1-o(1)$. It follows that $\dtv(\cC, t) = 1-o(1)$.
For the $L^2$ lower bound, it suffices to consider the representation $\nu := [n-1,1]$. We have $\bg \chi^\nu(\cC)\bd = \frac{f(\cC)-1}{n-1} \geq \frac{f(\cC)-1}{n}$. Since $n$ is large we have $\bg \chi^\nu(\cC)\bd \geq (f(\cC)/n)^{(1-\varepsilon/2)/(1-\varepsilon)}$ since $f(\cC)\to \infty$, and $d_\nu^2\geq n^2/2$ since $d_\nu \sim n$, and therefore as $n\to \infty$,
\begin{equation}
    2\dLtwo(\cC, t) \geq d_\nu^2 \bg \chi^\nu(\cC)\bd^{2t}  \geq d_\nu^2 \pg \frac{f(\cC)-1}{n}\pd^{2(1-\varepsilon) \frac{\ln n}{\ln (n/f(\cC))}}  \geq \frac{n^{2}}{2} n^{-2+\varepsilon} = \frac{n^{\varepsilon}}{2}\xrightarrow[n\to \infty]{} \infty. \qedhere
\end{equation}
\end{proof}

\subsection{An improved upper bound}
Since there is no error term in Theorem \ref{thm: character bound with no error nor helping term}, we can use the refined bounds on the Witten zeta function from \cite{TeyssierThévenin2025virtualdegreeswitten} to obtain more precise mixing time bounds within the cutoff \textit{window}. We write this in a similar form to \cite[Conjecture 9.3]{Saloff-Coste2004RandomWalksOnFiniteGroups}.
\begin{proposition}\label{prop: improved ell two upper bound}
    Let $\delta>0$, and let $n_0 = n_0(\delta)$ as in Theorem \ref{thm: character bound with no error nor helping term}. There exists a universal constant $A$ such that the following holds. Let $n\geq n_0$, let $\cC\ne \ag \Id \ad$ be a conjugacy class of $\kS_n$ with at least $\delta n$ fixed points, and let $a>0$ such that $t:= \frac{\ln n}{\ln(n/f(\cC))}\pg 1+ \frac{a}{\ln (n-1)} \pd$ is an integer. Then
    \begin{equation}
        2\dtv(\cC, t) \leq \dLtwo(\cC, t) \leq e^{-a} + Ae^{-3a/2}.
    \end{equation}
\end{proposition}
\begin{proof}
 The first inequality always holds by the Cauchy--Schwarz inequality. By the Diaconis--Shahshahani upper bound lemma (recalled as \eqref{eq: coset Diaconis--Shahshahani upper bound lemma}) and Theorem \ref{thm: character bound with no error nor helping term}, we have
    \begin{equation}\label{eq: improved ell two upper bound inter}
        2\dLtwo(\cC, t)^2 = \sum_{\lambda\in \widehat{\mathfrak{S}_n}^{**}} d_\lambda^2 \pg d_\lambda^{-\frac{\ln(n/f)}{\ln n}}\pd^{2t} = \sum_{\lambda\in \widehat{\mathfrak{S}_n}^{**}} d_\lambda^{-s_n},
    \end{equation}
where $s_n = 2a/\ln (n-1)$. For level 1 representations we have $(d_{[n-1,1]})^{-s_n} = (d_{[2,1^{n-2}]})^{-s_n}=(n-1)^{-2a/\ln(n-1)} = e^{-2a}$.
Set $\Lambda = \ag \lambda \in \widehat{\mathfrak{S}_n}^{**} \du 2\leq  \rsym(\lambda) \leq 18\ad$. Then by symmetry $|\Lambda|/2\leq A_0:=\sum_{i=2}^{18} p(i)$. Also, for each $\lambda\in \Lambda$ we have $d_\lambda \geq n(n-2)/2\geq (n-1)^{3/2}$ (since $n$ is large), so $d_\lambda^{-s_n} \leq e^{-3a}$. Therefore $\sum_{\lambda \in \Lambda} d_\lambda^{-s_n} \leq 2A_0 e^{-3a}$.
On the other hand, without loss of generality (up to increasing the value of $A$) we may assume that $2a\geq 12\ln 2$, so by \cite[Proposition 6.1]{TeyssierThévenin2025virtualdegreeswitten} we have $\sum_{\lambda \in \widehat{\mathfrak{S}_n}^{**}\du \rsym(\lambda) >18} d_\lambda^{-s_n} \leq 2e^{-3a}$. Combining these two bounds and plugging them into \eqref{eq: improved ell two upper bound inter}, we obtain $2\dLtwo(\cC, t)^2 \leq 2e^{-2a} +  2\pg A_0 + 1\pd e^{-3a}$. We conclude that
\begin{equation}
    \dLtwo(\cC, t) \leq \sqrt{e^{2a} + \pg A_0 + 1\pd e^{3a}} \leq e^{-a} + (A_0 + 1)^{1/2}e^{-3a/2}. \qedhere
\end{equation}
\end{proof}

\section*{Acknowledgements}

We thank Nathanaël Berestycki, Persi Diaconis, and Valentin Féray for stimulating conversations and helpful comments. LT is especially grateful to Persi Diaconis for his motivating support and for inspiring discussions over the years. 
SOT gratefully acknowledges the hospitality of the University of Vienna, particularly the probability group, during their multiple research visits. SOT was partially supported by EPSRC grant EP/N004566/1.
LT was supported by the Pacific Institute for the Mathematical Sciences, the Centre national de la recherche scientifique, and the Simons foundation, via a PIMS-CNRS-Simons postdoctoral fellowship.
LT and PT were supported by the Austrian Science Fund (FWF) under grants 10.55776/P33083 and SFB F 1002.

\bibliographystyle{alpha}
\bibliography{bibliographieLucas}

\end{document}